\documentclass[twoside,a4paper,reqno,11pt]{amsart}

\usepackage[margin=1in]{geometry}

\usepackage{amsfonts, amsmath, amssymb, mathrsfs, bm, latexsym, stmaryrd, array, mathtools, bold-extra}

\usepackage{color}
\usepackage{xcolor}

\usepackage[colorlinks,pagebackref]{hyperref}
\usepackage[pagebackref]{hyperref}

\usepackage{bbm}
\usepackage[all,cmtip]{xy}

\usepackage{enumerate}
\usepackage[shortlabels]{enumitem}

\newtheorem{theorem}{Theorem} 
\newtheorem*{conj*}{Conjecture}

\newtheorem{conj}[theorem]{Conjecture}

\newtheorem{thm}{Theorem}[section] 
\newtheorem{prop}[thm]{Proposition} 
\newtheorem{lem}[thm]{Lemma}
\newtheorem{cor}[thm]{Corollary}

\theoremstyle{definition}
\newtheorem{rem}[thm]{Remark}

\theoremstyle{definition}

\newtheorem*{example}{Example}

\theoremstyle{remark}

\DeclareMathOperator{\Aut}{Aut}

\DeclareMathOperator{\Irr}{Irr}

\DeclareMathOperator{\he}{ht}
\DeclareMathOperator{\de}{dl}

\DeclareMathOperator{\GL}{GL}

\DeclareMathOperator{\SL}{SL}

\DeclareMathOperator{\SO}{SO}

\DeclareMathOperator{\dlen}{dl}

\DeclareMathOperator{\hei}{ht}

\newcommand \ibr {{\textup{IBr}}_p}

\newcommand \opp {{\bf{O}}_{p}}
\newcommand \opd {{\bf{O}}_{p'}} 

\newcommand{\la}{\lambda}

\newcommand{\Syl}{{\mathrm {Syl}}}

\DeclareMathOperator{\Sp}{Sp}
\DeclareMathOperator{\Sym}{S}

\DeclareMathOperator{\PGL}{PGL}

\DeclareMathOperator{\GU}{GU}

\newcommand{\cd}{{\mathrm {cd}}}

\newcommand{\FF}{\mathbb{F}}

\newcommand{\Centralizer}{\mathbf{C}}
\newcommand{\cent}{\Centralizer}
\newcommand{\Center}{\mathbf{Z}}
\newcommand{\zent}{\Center}

\newcommand{\wt}{\widetilde}

\newcommand{\tw}[1]{{}^#1}
\numberwithin{equation}{section}

\newcommand{\Alt}{{\mathrm {A}}}

\newcommand{\eps}{\epsilon}

\newcommand{\bG}{\mathbf{G}}

\def\irr#1{{\rm Irr}(#1)}

\newcommand{\type}{\operatorname}

\begin{document}

\title[Derived Length and character degrees]{Sylow subgroups and the number of irreducible characters of degrees divisible by a prime $p$}

\author[James P. Cossey]{James P. Cossey}
\address{Department of Mathematics, University of Akron, Akron, OH 44325, USA}
\email{cossey@uakron.edu}

\author[M. L. Lewis]{Mark L. Lewis}
\address{Department of Mathematical Sciences, Kent State University, Kent, OH 44266, USA}
\email{lewis@math.kent.edu}

\author[A. A. Schaeffer Fry]{A. A. Schaeffer Fry}
\address{Department of Mathematics, University of Denver, Denver, CO 80210, USA}
\email{mandi.schaefferfry@du.edu}

\author[H. P. Tong-Viet]{Hung P. Tong-Viet}
\address{Department of Mathematics and Statistics, Binghamton University, Binghamton, NY 13902-6000, USA}
\email{htongvie@binghamton.edu}

\renewcommand{\shortauthors}{Cossey, Lewis, Schaeffer Fry, Tong-Viet}
\begin{abstract}
Let $G$ be a finite group and $p$ a prime. We establish an upper bound for the derived length of a Sylow $p$-subgroup of $G$ in terms of the number of irreducible characters of $G$ whose degrees are divisible by $p$. We also prove that if $B$ is a $p$-block of a finite $p$-solvable group $G$ with defect group $D$, then  the derived length of $D$ is at most one more than the number of ordinary irreducible characters of positive height in $B$.
\end{abstract}

\thanks{The third-named author is grateful for support from a CAREER grant from the U.S. National Science Foundation, DMS-2439897.}





\maketitle

\section{Introduction}

Let $G$ be a finite group and $p$ a prime. The It\^{o}--Michler theorem asserts that $p$ divides the degrees of no complex irreducible characters of $G$ if and only if $G$ possesses a normal abelian Sylow $p$-subgroup. Fix a Sylow $p$-subgroup $P$ of $G$, and let $n_p(G)$ denote the number of complex irreducible characters of $G$ whose degrees are divisible by $p$. We refer to such characters as \emph{$p$-singular characters}, and their degrees as \emph{$p$-singular character degrees}. We write $\Irr(G)$ for the set of complex irreducible characters of $G$, and $\cd(G)$ for the set of their degrees. 

From the It\^{o}--Michler theorem, if $n_p(G)=0$, then the derived length of $P$, denoted $\dlen(P)$, equals~$1$. In particular, in this case we have
$\dlen(P) = n_p(G) + 1.$

The finite groups satisfying $n_p(G)=1$ were completely classified by Guralnick \emph{et al.}~\cite{GGLMNT}. For such a group, we have $|\cd_p(G)|=1$, where $\cd_p(G)$ denotes the set of $p$-singular character degrees of $G$. Isaacs, Moret\'o, Navarro, and Tiep proved in~\cite{IMNT} that, for such groups $G$, the Sylow $p$-subgroup $P$ is metabelian; that is, $\dlen(P)\leq 2=|\cd_p(G)|+1$. In particular, they conjectured that 
$|\cd(P)| \leq |\cd_p(G)| + 1.$
Since $P$ is an $M$-group, Taketa's theorem (\cite[Corollary 5.13]{Isaacs}) yields $\dlen(P)\leq |\cd(P)|$. Moreover, it is clear that $|\cd_p(G)|\leq n_p(G)$.  Thus the above conjecture would imply that
\[
\dlen(P)\leq |\cd(P)| \leq  |\cd_p(G)| + 1 \leq n_p(G) + 1.
\]
Therefore, if the above conjecture holds, then $\dlen(P)\leq n_p(G)+1$. In \cite{FLZ}, the authors verify  Isaacs-Moret\'{o}-Navarro-Tiep's conjecture for symmetric groups, finite general linear or unitary groups.
Navarro also posed the following question  in \cite{Navarro_a1}: Is it true that $\dlen(P)\leq |\cd_p(G)|+1$?
In this paper, we prove the above weaker inequality.

\begin{theorem}\label{th:main}
Let $G$ be a finite group and let $p$ be a prime. Let $P$ be a Sylow $p$-subgroup of $G$. Then the derived length of $P$ is at most one more than the number of of irreducible characters of $G$ whose degrees are divisible by $p$; that is, $\dlen(P)\leq n_p(G)+1.$
\end{theorem}

We recall the Isaacs--Seitz conjecture (\cite[p. 206]{Isaacs}), which states that if $G$ is a finite solvable group, then
$\dlen(G)\leq |\cd(G)|.$
To prove Theorem \ref{th:main}, we first establish it for finite $p$-solvable groups and then reduce the general case to finite nonabelian simple groups. In particular, the following result is needed to complete the proof.

For a finite group $G$ and a prime $p$, let $n_p^*(G)$ denote the number of $\Aut(G)$-orbits on the set of all $p$-singular irreducible characters of $G$. Note that $$|\cd_p(G)| \le n_p^*(G) \le n_p(G).$$

\begin{theorem}\label{th:simple groups}
Let $G$ be a finite nonabelian simple group and $p$ a prime dividing $|G|$. Then, for any Sylow $p$-subgroup $P$ of $G$,
$\dlen(P) \le n_p^*(G).$
\end{theorem}

For sporadic simple groups, one can show that 
$
\dlen(P) \le |\cd_p(G)|,$
with equality occurring only in the case $(G,p)=(\mathrm{M}_{22},2)$. 
For alternating groups $\Alt_n$ with $n\ge 5$ and for simple groups of Lie type, we prove that $
\dlen(P) \le n_p^*(G).$
These results are established in Section~\ref{s:simple}.

We now turn to the theory of Brauer blocks. Let $p$ be a prime and let $B$ be a Brauer $p$-block of $G$ with defect group $D = D(B)$. Write $|P| = p^a$ and $|D| = p^d$, so that $d$ is called the \emph{defect} of the block $B$. Denote by $\Irr(B)$ the set of ordinary irreducible characters belonging to $B$. For each $\chi \in \Irr(B)$, the \emph{height} of $\chi$, denoted $\hei(\chi)$, is defined by the relation 
\[
\chi(1)_p = p^{\,a - d + \hei(\chi)},
\]
where, for a positive integer $n$, $n_p$ denotes the largest power of $p$ dividing $n$.  
Let $n_p(B)$ denote the number of characters $\chi \in \Irr(B)$ of positive height.

The \textbf{Brauer height zero conjecture}, recently settled in ~\cite{MNST}, asserts that for any $p$-block $B$ of $G$, the defect group $D$ is abelian if and only if every character in $\Irr(B)$ has height zero. Equivalently,
$n_p(B) = 0 $ if and only if $ \dlen(D) = 1.$
In particular, one always has $\dlen(D) \leq n_p(B) + 1$ in this case.
This inequality holds trivially when $B$ has $p$-defect zero (since $D = 1$). When $D$ has full defect, i.e. $d = a$, the quantity $n_p(B)$ counts the number of characters in $\Irr(B)$ whose degrees are divisible by $p$.

We therefore propose the following conjecture.

\begin{conj}\label{block}
Let $G$ be a finite group and $p$ a prime.  
Let $B$ be a Brauer $p$-block of $G$ with defect group $D$.  
Then there exists an absolute constant $k>0$, independent of $B,p$ and $D$, such that
$\dlen(D) \leq n_p(B) + k.$
\end{conj}

In our next result, we show that one may take $k=1$ when $G$ is $p$-solvable.

\begin{theorem}\label{maink}  Let $G$ be a finite $p$-solvable group, and let $B$ be a $p$-block of $G$ with defect group $D$.  Then $\dlen(D) \leq n_p(B) + 1.$
\end{theorem}

There is a related conjecture due to Navarro stating that 
$\dlen(D) \le |\cd(B)|,$
where $\cd(B) = \{\chi(1) : \chi \in \Irr(B)\}$. See \cite{GMS} for several results concerning this conjecture. However, it is not always the case that $|\cd(B)| \le n_p(B) + 1$ or vice versa.

\begin{example}
 Let $G=\Sym_9$ be the symmetric group of degree $9$ and let $p=2$. Then $G$ has two $p$-blocks,  say $B_0=B_0(G)$, the principal block with defect group $P$, where $P$ is a Sylow $2$-subgroup of $G$, and the block $B_1$ with  defect $4$. For the principal block, we have $|\cd(B_0)|=10$ and $n_2(B_0)+1=13$. For $B_1$, we have $|\cd(B_1)|=5$ and $n_2(B_1)+1=3.$
\end{example}

We denote by $\hei(B)$ the set of all character heights in $\Irr(B)$. It is clear that $|\hei(B)| \le |\cd(B)|$ and $|\hei(B)| \le n_p(B) + 1$. In \cite{GMS}, it was shown that if 
$G \in \{\Alt_n, \Sym_n, \GL_n(q)\},$
where $q$ is a power of $p$, then $\dlen(D) \le |\hei(B)|$ for every block $B$ of $G$ with defect group $D$. In particular, Conjecture~\ref{block} holds in these cases with $k=1$.  It is also true in the case when $G$ is a finite $p$-group.
Note, however, that the inequality $\dlen(D) \le |\hei(B)|$ does not hold in general;  counterexamples were constructed in \cite{GMS}.

A second related conjecture is the Eaton-Moret\'o conjecture from \cite{EaMo}.  For a finite $p$-group $P$ define $m (P)$ to be the minimal nonlinear degree of irreducible characters in $\irr P$ and set ${\rm mh} (P)$ by $m(P) = p^{{\rm mh} (P)}$.  Define ${\rm mh} (B)$ to be the minimum among positive numbers in $\hei (B)$. (Both ${\rm{mh}}(P)$ and ${\rm mh}(B)$ are defined to be $\infty$ if no such degree, resp. height, exists.)  The Eaton-Moret\'o conjecture suggests for all groups $G$ that a block $B$ with defect group $D$ should satisfy ${\rm mh} (B) = {\rm mh} (D)$.  This conjecture has been proved for $p$-solvable groups by Navarro in \cite{Navarro_2} and for general linear and unitary groups by Feng, Liu, and Zhang in \cite{FLZ1}.   In general, this conjecture is still open, although an important case was completed in \cite{MMR} and minimal counterexamples were studied in \cite{MSF}.

\medskip
Our paper is structured as follows.  In Section \ref{s:prelim}, we prove several preliminary results about derived lengths of certain $p$-groups.  In Section \ref{s:solv}, we prove Theorem \ref{th:main} for $p$-solvable groups, and then we show how to prove Theorem \ref{th:main} provided Theorem \ref{th:simple groups} is true.  Theorem \ref{th:simple groups} is proved in Section \ref{s:simple}.  This section is divided in subsections, first for the sporadic simple groups, second are the alternating groups, third are the exceptional groups of Lie type, fourth are the classical groups of Lie type in defining characteristic, and finally are the classical groups of Lie type in defining characteristic.  The proof of Theorem \ref{maink} is handled in Section \ref{s:block}.

\medskip\noindent\textbf{Notation}. 
Our notation is standard. We follow \cite{Isaacs} for character theory of finite groups and \cite{Navarro_b1} for block theory. Given finite groups $Y\leq X$ and $\chi\in\Irr(X)$ and $\psi\in\Irr(Y)$, we write $\chi_Y$ or sometimes $\chi|_Y$ for the restriction of $\chi$ to $Y$ and we write $\psi^X$ for the induced character of $X$. If $Y\lhd X$, we further write $\Irr(X|\psi)$ for the set of irreducible constituents of $\psi^X$ and $X_\psi$ for the stabilizer in $X$ of $\psi$.


\section{Preliminaries on derived lengths of \texorpdfstring{$p$}{p}-groups}\label{s:prelim}
In this section, we collect several results that will be used in the proof of the main theorems. We begin with the following upper bound for the derived length of finite $p$-groups.

\begin{lem}\label{lem1}
Let $P$ be a finite $p$-group of order $p^n$ for some prime $p$. If the derived length of $P$ is $k+1$, then $n\ge 2^{k}+2k-2.$ In particular, if $P$ is nonabelian, then $\de(P)\leq 1+\log_2n.$
\end{lem}

\begin{proof}
The first claim is Theorem 1 in \cite{Mann}. Assume that $P$ is nonabelian. Then $k\ge 1$ and so $n\ge 2^k$. Thus $k\leq \log_2n$ and the lemma follows.
\end{proof}

We next discuss the Sylow subgroups of symmetric and classical groups and obtain bounds on their derived lengths.
Throughout, we let $P_{p^m}$ denote a Sylow $p$-subgroup of the symmetric group $\Sym_{p^m}$ on $p^m$ letters. If $n$ has $p$-adic expansion $a_0+a_1p+a_2p^2+\cdots+a_kp^k$, where $0\leq a_i<p-1$ for $0\leq i\leq k-1$ and $0<a_k<p$, then the Sylow $p$-subgroups of $\Sym_n$ are of the form 
\[P=\prod_{i=0}^k (P_{p^i})^{a_i}.\]
Further, the group $P_{p^i}$ can be written $P_{p^{i-1}}\wr C_p$, where we note that $P_{p^0}=1$ and $P_{p^1}=C_p$.

From this structure, we see:
\begin{lem}\label{lem:dlsymmetric}
Let $P\in\Syl_p(\Sym_n)$ be as above. Then $\mathrm{dl}(P)=k$.
\end{lem}
\begin{proof}
A proof can be found, e.g. in \cite[Lem. 4.2]{GMS}, but follows from a result due to Kaloujnine (see \cite[Theorem III.15.3]{Huppert}).
\end{proof}

We next consider Sylow $p$-subgroups of $\GL_n(\epsilon q)$ when $p\nmid q$, which were described by Weir \cite{weir} for $p$ odd and by Carter--Fong \cite{carterfong} for $p=2$. (Here for $\eps\in\{\pm1\}$, we use the common notation $\GL_n(\eps q)$ for $\GL_n(q)$ if $\eps=1$ and $\GU_n(q)$ if $\eps=-1$. We will do the same for $\SL_n(\eps q)$ and $\operatorname{PSL}_n(\eps q)$.)  

First let $p$ be odd, and let $e$ be the order of $\epsilon q$ modulo $p$. Here we let $n=ew+r$, where $0\leq r<e$. Then a Sylow $p$-subgroup of $\GL_n(\epsilon q)$ is isomorphic to one of $\GL_{ew}(\epsilon q)$. If $w=a_0+a_1p+a_2p^2+\cdots+a_kp^k$, where $0\leq a_i<p-1$ for $0\leq i\leq k-1$ and $0<a_k<p$ is the $p$-adic expansion of $w$, then a Sylow $p$-subgroup of $\GL_{n}(\epsilon q)$ is of the form 
\[Q=\prod_{i=0}^k (Q_{p^i})^{a_i},\] where $Q_{p^i}$ is a Sylow $p$-subgroup of $\GL_{ep^{i}}(\epsilon q)$. Further, $Q_{p^i}\cong C_{p^b}\wr P_{p^i}$, where $C_{p^b}$ is the cyclic group of order $p^b:=((\epsilon q)^e-1)_p$ and $P_{p^i}$ is a Sylow $p$-subgroup of $\Sym_{p^i}$ as above. 

Now assume that $p=2$. The structure of $Q\in\Syl_2(\GL_n(\epsilon q))$  is similar to above, with $Q=\prod_{i=0}^k(Q_{2^i})^{a_i}$, where $n=\sum_{i=0}^k a_i2^i$ is the $2$-adic expansion with $a_i\in\{0,1\}$ and $a_k=1$. Here $Q_{2^i}=Q_2\wr P_{2^{i-1}}$ where $P_{2^{i-1}}\in\Syl_{2}(\Sym_{2^{i-1}})$ and $Q_2\in\Syl_2(\GL_2(\epsilon q))$. Further, $Q_2\cong C_{2^b}\wr C_2$ if $q\equiv \epsilon\pmod 4$, with $2^b=(q-\epsilon)_2$; and $Q_2$ is semidihedral of size $2^{b+2}$ if $q\equiv -\epsilon \pmod 4$, where $2^b=(q+\epsilon)_2$.

With this, we also have:
\begin{lem}\label{lem:dlGL}
Let $p$ be a prime not dividing $q$ and let $Q\in\Syl_p(\GL_n(\epsilon q))$ be as above. Then $\mathrm{dl}(Q)=k+1$.
\end{lem}
\begin{proof}
If $p$ is odd or $p=2$ with $q\equiv \epsilon\pmod 4$, this uses the exact same proof as \cite[Lem. 4.2]{GMS}. If $p=2$ with $q\equiv -\epsilon\pmod 4$, the same reasoning still gives the result, once one notes that a semidihedral group has derived length 2.
\end{proof}

Next, we note that the Sylow $p$-subgroups of $\SO_{2n+1}(q)$ or $\Sp_{2n}(q)$, where $p$ is an odd prime not dividing $q$, is the same as one of $\GL_{2n+1}(q)$ (and hence also of $\GL_{2n}(q)$) if $d:=d_p(q)$ is even and is the same as one of $\GL_n(q)$ if $d$ is odd. A Sylow $p$-subgroup of $\SO_{2n}^\pm(q)$ with $p\nmid q$ odd is either a Sylow $p$-subgroup of $\SO_{2n+1}(q)$ or of $\SO_{2n-1}(q)$. Hence in any case, there is some $m$ such that the Sylow $p$-subgroups of the classical group of type $\type{B}, \type{C},$ or $\type{D}$ under consideration is the same as one of $\GL_m(q)$.

If instead $p=2$, then the Sylow subgroups of orthogonal and symplectic groups are again described in \cite{carterfong}. In what follows, we let $\operatorname{O}_{2n+1}(q)$ and $\operatorname{O}_{2n}^\pm(q)$ denote the full orthogonal groups. Here the Sylow $2$-subgroups of $\operatorname{O}_{2n+1}(q)$ and  $\Sp_{2n}(q)$ are of the form $\prod_{i=0}^k (R_{2^i})^{a_i}$, where $R_{2^i}$ is a Sylow $2$-subgroup of the corresponding classical group $\operatorname{O}_{2^i+1}(q)$ or  $\Sp_{2^i}(q)$ and $2n=\sum_{i=0}^k a_i 2^i$ is the $2$-adic expansion of $2n$. Further, for $i\geq 2$, we have $R_{2^i}\cong R_{2}\wr P_{2^{i-1}}$, and $R_2$ has derived length $2$. 
The Sylow $2$-subgroups $R$ of $\operatorname{O}_{2n}^\pm(q)$ satisfy that $[R, R]=[R_0, R_0]$, where $R_0$ is a Sylow $2$-subgroup of either $\operatorname{O}_{2n+1}(q)$ or $\operatorname{O}_{2n-1}(q)$. 

From this and the same proof as before, we have:
\begin{lem}\label{lem:dlclassical2}
Let $q$ be an odd prime and let $n$ be an integer with $2n=a_0+a_12+\ldots+a_{k-1}2^{k-1}+2^k$ with $a_i\in\{0,1\}$ for $1\leq i\leq k-1$. Let 
$R\in\Syl_2(G)$, where $G=\Sp_{2n}(q)$, $\operatorname{O}_{2n+1}(q)$, or $\operatorname{O}_{2n}^\pm(q)$. Then $\mathrm{dl}(R)=k+1$ in the first two cases and $\mathrm{dl}(R)\in\{k+1, k\}$ in the latter. 
\end{lem}

\bigskip

We will also use the following related bounds, on the derived length of a finite $p$-group of $\GL_n(F)$, where $F$ is any field, and of $\GL_n(q)$, where $q$ is a prime power as before.
\begin{lem}\label{lem2}
Let $F$ be any field and let $P$ be a finite $p$-subgroup of $\GL_n(F)$. Then  $$\de(P)\leq \lfloor \log_2(n)\rfloor+1.$$ 
\end{lem}

\begin{proof}
This is Lemma 2.3 in \cite{MST}.
\end{proof}

\begin{lem}\label{lem3}
Let $P$ be a $p$-subgroup of $\GL_n(q)$, where $q$ is a prime power. Then  $$\de(P)\leq \lfloor \log_2(n/e)\rfloor+1$$ where $e=1$ if $p$ divides $q$, and otherwise, $e$ is the order of $q$ modulo $p.$
\end{lem}

\begin{proof} This is Lemma 2.8 in \cite{MST}.
\end{proof}


\section{Solvable and \texorpdfstring{$p$}{p}-solvable groups}\label{s:solv}

In this section, we prove the main theorem for finite $p$-solvable groups and reduce the proof of the main theorem to proving Theorem \ref{th:simple groups}. Recall that for a finite group $G$ and a prime $p$, $n_p (G)$ is the number of irreducible characters of $G$ that are divisible by $p$.

\begin{thm} \label{p-solvable}
Let $p$ be a prime, let $G$ be a finite $p$-solvable group, and let $P$ be a Sylow $p$-subgroup of $G$.  Then ${ \dlen} (P) \le n_p (G) + 1$.
\end{thm}
	
\begin{proof}
We work by induction on $|G|$.  Fix $n = n_p (G)$.  
If $n = 0$, then by the It\^o-Michler theorem, we know that $P$ is abelian (and normal in $G$), and so, the result holds.  Thus, we may assume that $n \ge 1$.  Note that this implies that $p$ divides $|G|$ and that $G$ is not abelian.  Let $N$ be a minimal normal subgroup of $G$ and we have that $1 < N < G$.  Notice that either $N$ is an elementary abelian $p$-group or $p$ does not divide $|N|$.  

We first suppose that $n_p (G/N) < n$.  Let $m = n_p (G/N)$, and note that $m \le n - 1$.  By the inductive hypothesis applied to $G/N$, we have that ${\rm dl} (PN/N) \le m + 1$.  Thus, $P^{(m+1)} \le N$.  Since $P \cap N$ is abelian, this implies that $P^{(m+2)} = 1$, and so, $P^{(n+1)} = 1$.  We have ${\rm dl} (P) \le n + 1$.

Next, suppose that $p$ does not divide $|N|$.  This implies that $N \cap P = 1$.  By the inductive hypothesis applied to $G/N$, we have that ${\rm dl} (PN/N) \le n + 1$.  Thus, $P^{(n+1)} \le N$.  Since $P \cap N  = 1$, we have $P^{(n+1)} = 1$, and thus, ${\rm dl} (P) \le n + 1$.

We are left with the case that $n_p (G/N) = n$ and $p$ divides $|N|$.  Since $G$ is $p$-solvable, this implies that $N$ is a $p$-group, and so, $N \le P$. We claim that $N \cap P' = 1$.   Suppose that $N \cap P' > 1$.  Then there exists a character $1 \ne \mu \in \irr {N \cap P'}$.  Let $\nu \in \irr {N \mid \mu}$, and suppose $\theta \in \irr {P \mid \nu}$. Since $\mu$ is a constituent of $\theta_{N \cap P'}$ and $\mu \ne 1$, we deduce that $P'$ is not contained in the kernel of $\theta$.  Hence, $p$ divides $\theta (1)$.  Suppose $\chi \in \irr {G \mid \theta}$ and observe that $\theta$ is a constituent of $\chi_P$, and so, $\nu$ is a constituent of $\chi_N$ and $\mu$ is a constituent of $\chi_{N \cap P'}$.  Let $\gamma$ be any irreducible constituent of $\chi_P$ and let $\delta$ be an irreducible constituent of $\gamma_{N}$.  We know that $\delta$ is conjugate to $\nu$ in $G$ since $N$ is normal in $G$ and $\delta$ and $\nu$ are both constituents of $\chi_N$.  We may write $\delta = \nu^g$ for some $g \in G$.   As $N$ is abelian, $\nu$ is an extension of $\mu$.  Thus, $\delta$ is an extension of $\mu^g$.  Now, $\mu^g$ is nontrivial, so $P'$ is not contained in the kernel of $\gamma$.  This implies that $p$ divides $\gamma (1)$.  Since $\gamma$ was arbitrary, we see that $p$ divides the degrees of all of the irreducible constituents of $\chi_P$, and this implies that $p$ divides $\chi (1)$.  This is a contradiction since all of the irreducible characters of $G$ whose degrees are divisible by $p$ have $N$ in their kernels.  Thus, the claim is proved and $N \cap P' = 1$.

By the inductive hypothesis applied to $G/N$, we have ${\rm dl} (PN/N) \le n + 1$.  Thus, $P^{(n+1)} \le N.$   Since $P' \cap N  = 1$, we have $P^{(n+1)} = 1$, and thus, ${\rm dl} (P) \le n + 1$.  This proves the theorem.
\end{proof}

In fact, we can generalize this further.  Fix a prime $p$.  We say that a finite group $G$ is a $\mathcal {C}_p$-group if all of the nonsolvable composition factors satisfy the hypothesis that  the number of orbits of the irreducible characters whose degree is a multiple of $p$ under their automorphism group is an upper bound for the derived length of their Sylow $p$-subgroups.  

\begin{thm}\label{mathcal C}
Let $p$ be a prime, let $G$ be a $\mathcal {C}_p$-group, and let $P$ be a Sylow $p$-subgroup of $G$.  Then ${\rm dl} (P) \le n_p (G) + 1$.
\end{thm}	

\begin{proof}
We work by induction on $|G|$.  By Theorem \ref{p-solvable}, we may assume that $G$ is not $p$-solvable.  Fix $n = n_p (G)$.  If $G$ is nonabelian simple, then since we are assuming that $G$ is a $\mathcal {C}_p$-group, we have a stronger result than the desired conclusion, so we are good.  Also, if $n = 0$, then by the It\^o-Michler theorem, we know that $P$ is abelian (and normal in $G$), and so, the result holds.  Thus, we may assume that $n \ge 1$.  Note that this implies that $p$ divides $|G|$ and that $G$ is not abelian.  Let $N$ be a minimal normal subgroup of $G$ and we have that $1 < N < G$.  
	
We first suppose that $n_p (G/N) < n$.  Let $m = n_p (G/N)$, and note that $m \le n - 1$.  By the inductive hypothesis applied to $G/N$, we have that ${\rm dl} (PN/N) \le m + 1$.  Thus, $P^{(m+1)} \le N$.  If $N$ is abelian, then ${\rm dl} (P) \le m+2 \le n+1$.  Thus, we may assume that $N$ is nonabelian.  We know that $N$ is a direct product of copies of some nonabelian finite simple group $S$.  In particular, $N = S^k$ for some positive integer $k$.  Let $Q$ be a Sylow $p$-subgroup of $S$, and we know that a Sylow subgroup of $N$ is isomorphic to $Q^k$.    Notice that we have at most $n-m$ orbits of irreducible characters of $N$ whose degrees are divisible by $p$.  This yields at most $n-m$ orbits of irreducible characters of $S$ whose degrees are divisible by $p$.  Since we are assuming that $G$ is a $\mathcal {C}_p$-group, this implies that ${\rm dl} (Q) \le n-m$.  Since $Q$ and $Q^k$ have the same derived length, we obtain ${\rm dl} (Q^k) \le n - m$.   This yields $(P^{(m+1)})^{(n-m)} = P^{(n-m+m+1)} = P^{(n+1)} = 1$.  We deduce that ${\rm dl} (P) \le n + 1$.  
	
Next, suppose that $p$ does not divide $|N|$.  This implies that $N \cap P = 1$.  By the inductive hypothesis applied to $G/N$, we have that ${\rm dl} (PN/N) \le n + 1$.  Thus, $P^{(n+1)} \le N$.  Since $P \cap N  = 1$, we have $P^{(n+1)} = 1$, and thus, ${\rm dl} (P) \le n + 1$.
	
We are left with the case $n_p (G/N) = n$ and $p$ divides $|N|$.  This implies that there are no irreducible characters of $N$ whose degree is divisible by $p$.  By the It\^o-Michler Theorem, this implies that $N$ has a normal abelian Sylow $p$-subgroup.  Since $N$ is either elementary abelian or the direct product of copies of a nonabelian simple group, we conclude that $N$ must be an elementary abelian $p$-group.  Let $1_N \ne \nu \in \irr N$, and let $\chi \in \irr {G \mid \nu}$.  We see that $p$ does not divide $\chi (1)$; so if $T$ is the stabilizer of $\nu$ in $G$, then $p$ does not divide $|G:T|$.  Hence, $T$ contains a Sylow $p$-subgroup $P$ of $G$.  Now, we know that $p$ does not divide $\gamma (1)/\nu (1)$ for all $\gamma \in \irr {T \mid \nu}$.  By the Navarro-Tiep theorem (\cite[Theorem A]{NT}), this implies that $P/N$ is abelian.  Hence, ${\rm dl} (P) \le 2$ which  meets the bound since $n \ge 1$.  This proves the result.
\end{proof}


\section{Finite simple groups}\label{s:simple}

The aim of this section is to prove Theorem \ref{th:simple groups}. That is, we show that every finite nonabelian simple group is a $\mathcal{C}_p$-group, as defined before Theorem \ref{mathcal C}.

Let $G$ be a finite group and let $p$ be a prime. Denote  by $\Irr_p(G)$ the set of irreducible characters of $G$ whose degrees are divisible by $p$. Define $$\cd_p(G)=\{\chi(1):\chi\in\Irr_p(G)\}$$ as the set of degrees of such characters.  For brevity, we call characters in $\Irr_p(G)$ $p$-singular characters. As before, let $n_p(G)=|\Irr_p(G)|$, and let $n_p^*(G)$ denote the number of orbits in  $\Irr_p(G)$ under the action of the automorphism group $\Aut(G)$ of $G$. Since  all characters within the same $\Aut(G)$-orbit have the same degrees, it is clear that $$|\cd_p(G)|\leq n_p^*(G).$$  

We follow  the notation and terminology of \cite{Navarro_b1} for block theory. Suppose $p\mid |G|$, and let  $B$ be a Brauer $p$-block of $G$ with defect group $D=D(B)$. Write $|G|_p=p^a$ and $|D|=p^d$, so that the defect of the block $B$ is the nonnegative integer $d=d(B)$. For each $\chi\in\Irr(B)$, the degree $\chi(1)$  satisfies $$\chi(1)_p=p^{a-d+\he(\chi)},$$ where $\he(\chi)$ denotes the height of $\chi$ (see \cite[Corollary 3.17]{Navarro_b1}).
Note that if $B$ is a block of $G$ with defect $d<a$, then  $p^{a-d}$ divides $\chi(1)$ for all $\chi\in\Irr(B)$. In particular, each such degree is divisible by  $p$.

\subsection{Sporadic groups}

We begin by proving Theorem \ref{th:simple groups} for the sporadic simple groups. 
\begin{lem}\label{lem:sporadic}
Let $G$ be a sporadic simple group or the Tits group $\tw{2}\type{F}_4(2)'$, and let $p$ be a prime dividing $|G|$. Let $P$ be a Sylow $p$-subgroup of $G$. Then $\de(P)\leq |\cd_p(G)|$ with equality if and only if $(G,p)=({\rm M}_{22},2)$.
\end{lem}

\begin{proof}
We first observe that if $p$ divides $|G|$ and a Sylow $p$-subgroup $P$ of $G$ is abelian, then $\de(P)=1$. Moreover, since $G$ is nonabelian simple, $P$ is not normal in $G$, and thus by the It\^{o}-Michler theorem (\cite[Theorem 7.1]{Navarro_b2}), we have $|\cd_p(G)|\ge 1$. Therefore, $\de(P)\leq |\cd_p(G)|$ in this case. We now  assume that $P$ is nonabelian, so  in particular, $|P|\ge p^3.$ 

Assume first that $G$ is not one of the simple groups in  $ \{\textrm{Th},\textrm{Ly},\textrm{Fi}_{24}',\textrm{B},\textrm{M}\}$. We can compute the derived length of the Sylow $p$-subgroup of $G$ using MAGMA \cite{magma} and compute the number $|\cd_p(G)|$ of $p$-singular character degrees  using GAP \cite{GAP}. For  all such cases with $\de(P)\ge 2$, we list in Table \ref{Tab1} the values of $\de(P),$ $|\cd_p(G)|$ and $\log_p(|P|)=\log_p(|G|_p)$. Note that all Sylow subgroups of $\textrm{J}_1$ are abelian, so $\textrm{J}_1$ does not appear in the table.

We now handle the exceptional cases where computational data is less directly available.

\medskip
(a) $G=\textrm{Th}$.   

Case $p=2:$  $|\cd_2(G)|=26$ and $|P|=2^{15}$. By Lemma \ref{lem1}, we have $$\de(P)\leq 1+\log_2 (15)\leq 5<|\cd_2(G)|.$$

Case $p=3:$  $|\cd_3(G)|=25$ and $|P|=3^{10}$, hence $\de(P)\leq 1+\log_2(10)\leq 5<|\cd_3(G)|.$

Case $p=5:$  $\cd_5(G)=22$ and $|P|=5^3$, so  $\de(P)\leq 1+\log_2 3\leq 3<|\cd_5(G)|.$

\medskip
(b)  $G=\textrm{Ly}$. 

Case $p=2:$ $|\cd_2(G)|=28$ and $|P|=2^8$, so $\de(P)\leq 1+\log_2(8)<28=|\cd_2(G)|.$

Case  $p=3:$ 
$|\cd_3(G)|=15$ and $|P|=3^7$, so  $\de(P)\leq 1+\log_2(7)\leq 4<28=|\cd_2(G)|.$

Case $p=5:$  $|\cd_5(G)|=22$ and $|P|=5^6$, thus   $\de(P)\leq 1+\log_2(6)<|\cd_5(G)|.$

\medskip
(c) $G=\textrm{Fi}_{24}'$. 

Case $p=2:$ $|\cd_2(G)|=63$ and $|P|=2^{21}$ and clearly $$|\cd_2(G)|=63>6\ge 1+\log_2(21)\ge \de(P).$$

Case $p=3:$ $|\cd_3(G)|=54$ and $|P|=3^{16}$; $\de(P)\leq 1+\log_2(16)=5<54.$

Case $p=7:$ $|\cd_7(G)|=71$ and $|P|=7^{3}$, so $\de(P)\leq 1+\log_2(3)\leq 3<71.$

\medskip
(d) $G=\textrm{B}$. 

We have $(p,|\cd_p(G)|,\log_p(|P|))=(2,110,41), (3,142,13)$ or $(5,143,6)$. In each case, Lemma \ref{lem1} gives $\de(P)\leq 1+\log_2(a)< |\cd_p(G)|$, where $|P|=p^a.$

\medskip
(e) $G=\textrm{M}$. We need to consider the cases: $p\in \{2,3,5,7,13\}$. Let $|P|=p^a$. Then $(p,|\cd_p(G)|,a)$ is one of the following: \[(2,115,46), (3,103,20), (5,104,9),(7,129,6),(13,122,3).\]
In each case, we verify that $\de(P)<|\cd_p(G)|$.

Finally, the equality $\de(P)=|\cd_p(G)|$ occurs  if and only if $(G,p)=(\textrm{M}_{22},2)$ as verified using the argument above together with Table \ref{Tab1}.
The proof is now complete.
\end{proof}

\begin{table}[htp]
\caption{Derived lengths of Sylow $p$-subgroups and the number of $p$-singular character degrees of sporadic simple groups}
\begin{center}
\begin{tabular}{c|cccc|c|cccc}\hline
$G$ & $p$ & $\de(P)$ & $|\cd_p(G)|$ &$\log_p(|P|)$ &$G$ & $p$ & $\de(P)$ & $|\cd_p(G)|$ &$\log_p(|P|)$\\\hline

$\textrm{M}_{11}$ & $2$ & $2$ & $4$ & $4$ & $\textrm{He}$ & $2$ & $3$ & $14$ & $10$\\

$\textrm{M}_{12}$ & $2$ & $2$ & $6$ & $6$ & & $3$ & $2$ & $14$ & $3$\\

 & $3$ & $2$ & $6$ & $3$ && $7$ & $2$ & $10$ & $3$\\

$\textrm{M}_{22}$ & $2$ & $3$ & $3$ & $7$ & $\textrm{HN}$ & $2$ & $3$ & $30$ & $14$  \\

$\textrm{M}_{23}$ & $2$ & $3$ & $6$ & $7$ &  & $3$ & $2$ & $28$ & $6$ \\

$\textrm{M}_{24}$ & $2$ & $3$ & $8$ & $10$&  & $5$ & $3$ & $27$ & $6$ \\

 & $3$ & $2$ & $12$ & $3$ & $\textrm{Fi}_{22}$ & $2$ & $4$ & $38$ & $17$ \\
 
 $\textrm{HS}$ & $2$ & $3$ & $11$ & $9$ &  & $3$ & $3$ & $39$ & $9$\\
 
 & $5$ & $2$ & $8$ & $3$ & $\textrm{Fi}_{23}$ & $2$ & $4$ & $68$ & $18$\\
 
 $\textrm{J}_{2}$ & $2$ & $3$ & $10$ & $7$ &  & $3$ & $3$ & $64$ & $13$\\
 
  & $3$ & $2$ & $10$ & $3$ & $\textrm{ON}$ & $2$ & $3$ & $14$ & $9$ \\
  
  $\textrm{Co}_{1}$ & $2$ & $4$ & $67$ & $21$&  & $7$ & $2$ & $8$ & $3$ \\
  
 & $3$ & $3$ & $72$ & $9$ & $\textrm{J}_3$ & $2$ & $3$ & $9$ & $7$ \\
  
 & $5$ & $2$ & $75$ & $4$&  & $3$ & $2$ & $8$ & $5$ \\
 
  $\textrm{Co}_{2}$ & $2$ & $4$ & $27$ & $18$ & $\textrm{Ru}$ & $2$ & $4$ & $20$ & $14$ \\ 
  
 & $3$ & $2$ & $27$ & $6$ &  & $3$ & $2$ & $19$ & $3$ \\
  
 & $5$ & $2$ & $31$ & $3$ & & $5$ & $2$ & $13$ & $3$ \\
 
 $\textrm{Co}_{3}$ & $2$ & $3$ & $22$ &$10$ & ${}^2\textrm{F}_4(2)'$ & $2$ & $3$ & $9$ & $11$ \\ 
  
 & $3$ & $2$ & $9$ & $7$ &  & $3$ & $2$ & $8$ & $3$  \\ 
  
 & $5$ & $2$ & $18$ & $3$\\
 
$\textrm{McL}$ & $2$ & $3$ & $11$ & $7$ \\
  
 & $3$ & $2$ & $9$ & $6$ \\
  
 & $5$ & $2$ & $7$ & $3$ \\ \hline
 
\end{tabular}
\end{center}
\label{Tab1}
\end{table}%

\subsection{Alternating groups}

We now consider the alternating groups. Let $n\ge 5$ and let $\Omega=\{1,2,\dots,n\}$.  Let $\Sym_n$ and $\Alt_n$ denote the symmetric group and alternating group on $\Omega$, respectively. Let $p$ be a prime dividing $|A_n|$. Then $2\leq p\leq n.$ Set $G=\Alt_n$, and let $\widetilde{P}$ be a Sylow $p$-subgroup of $\Sym_n$. Then $P=G\cap\widetilde{P}$ is a Sylow $p$-subgroup of $G$. 
From Lemma \ref{lem:dlsymmetric}, we know that $$\de(\widetilde{P})\leq \lfloor \log_p(n)\rfloor.$$ Since $P\leq \widetilde{P}$, it follows that $\de(P)\leq \de(\widetilde{P})$. Therefore, 
\[\de(P)\leq \lfloor \log_p(n)\rfloor.\]

We refer the reader to \cite{James}, \cite{JK} or \cite{Olssonb} for background information on the representation theory of symmetric groups.
We fix some notation. Let $\lambda=(\lambda_1,\lambda_2,\dots,\lambda_r)$ be a finite sequence of positive integers such that $\lambda_1\ge \lambda_2\ge \cdots \ge\lambda_r$. Define $|\lambda|=\lambda_1+\lambda_2+\cdots+\lambda_r$, and we say that $\lambda$ is a partition of $|\lambda|$. The conjugate partition of $\lambda$ is denoted by $\lambda'$. The Young diagram associated to $\lambda$ is the set $$[\lambda]=\{(i,j)\in\mathbb{N}^2:1\leq i\leq r,1\leq j\leq \lambda_i\}.$$ For a given $(i,j)\in[\lambda]$, we denote by $H_{i,j}(\lambda)$ the corresponding $(i,j)$-hook and let $h_{i,j}(\lambda)=|H_{i,j}(\lambda)|$. A $p$-hook of $\lambda$ is a hook of length $p$, that is, $h_{(i,j)}(\lambda)=p$. A partition $\lambda$  is called a $p$-core partition if it has no $p$-hook. For each partition $\lambda$, we can repeatedly remove   $(i,j)$-rim hook of length $p$, where $H_{i,j}(\lambda)$ is a $p$-hook, to obtain a unique $p$-core partition, denoted  $\lambda_{(p)}$. The numbers of $p$-hooks removed  in this process is called the $p$-weight of $\lambda$, denoted  $w_p(\lambda)$.  We then have: \[n=|\lambda|=|\lambda_{(p)}|+pw_p(\lambda).\]

The irreducible characters of $\Sym_n$ are labeled by partitions of $n$. For a partition $\lambda$ of $n$, let $\chi^{\lambda}$ denote the corresponding irreducible character of $\Sym_n$.  If $\lambda\neq \lambda'$, then $\chi^{\lambda}\vert_{\Alt_n}=\chi^{\lambda'}\vert_{\Alt_n}$ remains irreducible, and hence $\{\chi^{\lambda}\vert_{\Alt_n}\}$ forms an $\Sym_n$-orbit of irreducible characters of $\Alt_n$. In particular, if $p$ divides $\chi^{\lambda}(1)$, then its restriction to $\Alt_n$ is also an irreducible character of degree divisible by $p$. If $\lambda=\lambda'$, that is, $\lambda$ is self-conjugate, then $\chi^\lambda\vert_{\Alt_n}$  decomposes into a sum of two irreducible characters, denoted $\chi^{\lambda\pm}$ which have equal degrees. The pair $\{\chi^{\lambda\pm}\}$ forms an $\Sym_n$-orbit of irreducible characters of $\Alt_n$.  In the case  $p$ is odd, both $\chi^{\lambda\pm}$ have degree divisible by $p$ whenever $p\mid \chi^{\lambda}(1)$. When $p=2$, if $a-d\ge 2$ (in the context of block theory), then $4\mid\chi^{\lambda}(1)$, and thus each $\chi^{\lambda\pm}$ has even degree, as wanted.

By Nakayama's conjecture (see \cite[6.2.21]{JK}), two irreducible characters $\chi^{\alpha}$ and $\chi^{\beta}$ of $\Sym_n$ labeled by two partitions $\alpha$ and $\beta$ of $n$ lie in the same $p$-block $B$ of $\Sym_n$ if and only if they have the same $p$-core, that is $\alpha_{(p)}=\beta_{(p)}$. Suppose $B$ has positive $p$-weight $w>0$,  and let $D$ be the defect group of $B$. Then: $|D|=p^d=|\Sym_{pw}|_p$ (see \cite[Prop. 11.3]{Olssonb}).  If $|\alpha_{(p)}|\ge p$, then $d<a$, where $|\Sym_n|_p=p^a$ and hence $p^{a-d}$ divides $\chi^\alpha(1)$. Furthermore, if $|\alpha_{(p)}|\ge 2p$, then $a-d\ge 2$, and so $p$ divides $\chi^{\alpha\pm }(1)$ in the self-conjugate case.

To obtain a lower bound on the number of $\Sym_n$-orbits of $p$-singular irreducible characters of $\Alt_n$, we proceed as follows:
We first identify  various $p$-core partitions  $\mu$ such that $|\mu|\ge p$ and construct as many as possible partitions  $\lambda$ of $n$ in such a way that $\lambda_{(p)}=\mu,$ and $n=|\mu|+pw$ with $w=w_p(\lambda)>0$. 

For a real number $x$, we frequently use the following inequalities $x\ge \lfloor x\rfloor> x-1$, where $\lfloor x\rfloor$ denotes the floor function.
\begin{lem}\label{lem:Alternating}
Let $G=\Alt_n$ with $n\ge 5.$ Let $p$ be a prime divisor of $|G|$ and let $P$ be a Sylow $p$-subgroup of $G$. Then $\de(P)\leq n_p^*(G)$.

\end{lem}

\begin{proof}
If $5 \leq n \leq 24$, then using \textsc{MAGMA} \cite{magma}, one can verify that $\de(P) \leq |\cd_p(G)|$, and the lemma follows. Thus, assume $n \geq 25$.
Since $\Aut(G) \cong \Sym_n$, it follows that the number of $\Aut(G)$-orbits of $p$-singular irreducible characters of $G$ equals the number of $\Sym_n$-orbits of such characters. Let $p$ be a prime divisor of $|G|$, so $2 \leq p \leq n$. Let $P$ be a Sylow $p$-subgroup of $G$.

\medskip
Assume $P$ is abelian.  
By the It\^{o}-Michler theorem, since $G$ is simple and $P$ is not normal, there exists an irreducible character of $G$ whose degree is divisible by $p$. Thus, $|\cd_p(G)| \geq 1 = \de(P)$, and the lemma holds.

\medskip
Assume $P$ is nonabelian.  
We now construct enough $p$-singular irreducible characters of $\Sym_n$ to show $n_p^*(G) \geq \de(P)$.

Let $n = a_0 + a_1 p + \cdots + a_t p^t$ be the $p$-adic expansion of $n$ with $0 \leq a_i < p$ and $a_t > 0$. Since $P$ is nonabelian, we have $t \geq 2$, so $n \geq p^2 + r$, where $r = n \bmod p$.
Recall from above that
$
\de(P) \leq \lfloor \log_p(n) \rfloor\leq \log_p(n).$
We proceed case by case based on the value of $p$.

\medskip
 (1) Assume that $p \geq 5$.

\smallskip
(a) Assume that $1 \leq r \leq p - 1$. Write $n = m p + r$ for some integer $m \geq 1$ and fixed remainder $r$. We will construct many pairwise distinct, non-conjugate partitions of $n$ with $p$-core partitions of size at least $p + r$, so that the degrees of the corresponding irreducible characters of $\Alt_n$ are divisible by $p$. This will give a lower bound for the number $ n_p^*(G)$.

By \cite[Theorem 1]{GO}, for each $ j = 1, 2, \dots, p - 1 $, there exists a $p$-core partition $ \mu_j $ with size $ |\mu_j| = pj + r $. These will serve as cores for families of partitions of $n$ with positive $p$-weight.
For example, for $ j = 1 $, take $ \mu_1 = (p, 1^r) $, a hook partition of size $ p + r $.

We fix $j=1$, and construct partitions $ \alpha_k $ and $ \beta_k $ as follows:

For each $ 0 \leq k \leq  \lfloor ({n - 1 - 2r})/({2p})   \rfloor $, define the hook partition
$\alpha_k = (n - r - kp, 1^{r + kp}).$
This is a partition of $ n $, and it has $p$-core $ \mu_1 $ and positive $p$-weight.

For $ 1 \leq k \leq  \lfloor ({n - 1 - r})/({2p})   \rfloor $, define
$\beta_k = (n - r - pk, 1 + pk, 1^{r - 1}).$
Again, this has size $n$ and $p$-core $ \mu_1 $ with positive $p$-weight.

These partitions are all distinct and  for each $ \lambda \in \{ \alpha_k, \beta_k \} $, either $\lambda$ is self-conjugate or its conjugate $ \lambda' \notin \{ \alpha_k, \beta_k \} $, ensuring that the restrictions of the corresponding characters of $ \Sym_n $ to $ \Alt_n $ give distinct irreducible $p$-singular characters.

Now, for $ j = 2, \dots, p - 1 $, use the $p$-core partitions $ \mu_j $ of size $ pj + r $ to construct new partitions $ \gamma_j $ by adding $n-r-pj$ nodes to the first row of the Young diagram of $\mu_j.$
This yields a partition of $n$ with $p$-core $ \mu_j $ and positive $p$-weight.

Let $ S $ denote the set of all $ \alpha_k $, $ \beta_k $, and $ \gamma_j $ constructed above. Since each partition in $S$ yields one irreducible $p$-singular character of $ \Sym_n $, we obtain
$n_p^*(G) \geq |S|.$
We have
\[
|S| =  \lfloor \frac{n - 1 - 2r}{2p}   \rfloor + 1 +  \lfloor \frac{n - 1 - r}{2p}   \rfloor + (p - 2) > \frac{2n - 2 - 3r}{2p} + (p - 3) = \frac{n}{p} - \frac{3r + 2}{2p} + (p - 3).
\]

Since $ 1 \leq r \leq p - 1 $, we have:
\[
\frac{3r + 2}{2p} \leq \frac{3(p - 1) + 2}{2p} = \frac{3p - 1}{2p} < 2.
\]
Therefore,
$
n_p^*(G)\ge |S| \geq {n}/{p} + p - 5.
$
In particular, $ n_p^*(G) \ge {n}/{p} $.
We know from earlier that
\[
\de(P) \leq \lfloor \log_p(n) \rfloor \leq \log_p(n).
\]
Thus it suffices to show
$
\log_p(n) \leq {n}/{p}.
$
Let $ x = \log_p(n) $, so that $ n = p^x $. Since $n\ge p^2+r$, $x\ge 2.$ The inequality above becomes:
\[
x \leq p^{x - 1}, \quad \text{or} \quad f(x) := p^{x - 1} - x \geq 0.
\]

We now verify this inequality.  If $ x = 2 $, then $ f(2) = p - 2 \geq 3 $ since $ p \geq 5 $.
We compute $ f'(x) = p^{x - 1} \log p - 1 > 0 $ for all $ x \geq 2 $, so $ f(x) $ is increasing.
Hence, for all $ x \geq 2 $, $ f(x) \geq f(2) > 0 $, which implies the desired inequality:
$
\log_p(n) \leq {n}/{p}.
$
Therefore
\[
\de(P) \leq \log_p(n) \leq \frac{n}{p} < n_p^*(G),
\]
and the claim follows.

\medskip
(b) Now assume $ r = 0 $, i.e., $ p \mid n $. Let $ \mu := (p - 2, 2) $ be a $p$-core partition of size $ p $.
For $ 0 \leq k \leq {n}/{p} - 1 $, let
$\alpha_k = (n - 2 - pk, 2, 1^{pk}).$
For $ 1 \leq k \leq  \lfloor ({n - 4})/({2p})   \rfloor $, let
$\beta_k = (n - 2 - pk, 2 + pk).$
These are all partitions of $ n $ with $p$-core $ \mu $, and they are pairwise distinct and non-conjugate.
Since $n\ge p^2,$ $n-4\ge p^2-p=p(p-1)\ge 4p$,  and hence $$\lfloor (n-4)/(2p)\rfloor\ge (n-4)/(2p)-1\ge 2-1=1.$$ Thus
$
n_p^*(G) \geq {n}/{p} \geq \log_p(n) \geq \de(P),
$
as required.

\medskip
(2) Assume $p = 3$ and let $n \equiv s \pmod{3}$, where $s \in \{1, 2, 3\}$.

\smallskip
(a) Assume $s = 1$.  
Let $\mu_1 = (3,1)$, which is a $3$-core partition of $4$. For $1 \leq k \leq  \lfloor {(n-4)}/{3}   \rfloor$, define
\[
\alpha_k = (n - 1 - 3k, 1^{3k + 1}).
\]
Each $\alpha_k$ is not self-conjugate, and $\alpha_k'$ is not equal to any $\alpha_\ell$ in the same family. We also define
\[
\beta_k = (n - 1 - 3k, 3k + 1), \quad \text{for } 1 \leq k \leq  \lfloor (n-2)/{6}   \rfloor.
\]
The partitions $\alpha_k$ and $\beta_k$ are pairwise distinct, non-conjugate, and all have $3$-core $\mu_1$. Hence, each corresponds to a distinct $\Sym_n$-orbit of $3$-singular irreducible characters of $\Alt_n$, so
\[
n_3^*(G) \geq  \lfloor \frac{n-4}{3}   \rfloor + 1 +  \lfloor \frac{n-2}{6}   \rfloor \geq \frac{n-1}{3} + \frac{n-2}{6} - 1.
\]
Since $n \geq 24$, we obtain
$
n_3^*(G) \geq {n}/{3} + {(n - 10)}/{6} \geq {n}/{3} + 2.
$
Because $\de(P) \leq \log_3(n)$ and $\log_3(n) \leq {n}/{3} + 2$ for all $n \geq 24$, it follows that $\de(P) \leq n_3^*(G)$.

\smallskip
(b) Assume $s = 2$.  
Let $\mu_1 = (3,1^2)$ and $\mu_2 = (4,2,1^2)$, which are $3$-core partitions of $5$ and $8$, respectively. Define:
\[
\alpha_k = (n - 2 - 3k, 1^{2 + 3k}), \quad \text{for } 0 \leq k \leq  \lfloor ({n - 4})/{6}   \rfloor,
\]
\[
\beta_k = (n - 2 - 3k, 3k + 2, 1), \quad \text{for } 1 \leq k \leq  \lfloor ({n - 3})/{6}   \rfloor,
\]
\[
\gamma_l = (n - 4 - 3l, 2, 1^{2 + 3l}), \quad \text{for } 0 \leq l \leq  \lfloor ({n - 8})/{6}   \rfloor.
\]
These are distinct, non-conjugate partitions of $n$ with $3$-core $\mu_1$ or $\mu_2$. As  $n \geq 24$, we get
\[
n_3^*(G) \geq  \lfloor \frac{n-4}{6}   \rfloor + 1 +  \lfloor \frac{n-3}{6}   \rfloor + \lfloor \frac{n-8}{6}   \rfloor+1  \geq \frac{3n - 15}{6} -1=\frac{n}{3}+\frac{n-21}{6} \geq \frac{n}{3}.
\]
Again, since $\de(P) \leq \log_3(n) \leq {n}/{3}$ for $n \geq 24$, we have $\de(P) \leq n_3^*(G)$.

\smallskip
(c) Assume $s = 3$ (i.e., $3 \mid n$).  
Let $\mu = (4,2)$ be a $3$-core partition of $6$. Define:
\[
\alpha_k = (n - 2 - 3k, 2, 1^{3k}), \quad \text{for } 0 \leq k \leq  \lfloor \frac{n}{3}   \rfloor - 2,
\]
\[
\beta_k = (n - 2 - 3k, 2 + 3k), \quad \text{for } 1 \leq k \leq  \lfloor \frac{n - 4}{6}   \rfloor.
\]
Each $\alpha_k$ and $\beta_k$ is a partition of $n$ with $3$-core $\mu$. Hence,
\[
n_3^*(G) \geq  ( \frac{n}{3} - 2   ) + 1 +  \lfloor \frac{n - 4}{6}   \rfloor \geq \frac{n}{3} - 1 + \frac{n - 4}{6} - 1 = \frac{n}{3} + \frac{n - 10}{6}.
\]
Since $n \geq 24$, we have
$
n_3^*(G) \geq {n}/{3} + 1,
$
and since $\de(P) \leq \log_3(n) \leq {n}/{3} + 1$, it follows that $\de(P) \leq n_3^*(G)$ as desired.

\medskip

(3) Assume $p = 2$. We distinguish the cases where $n$ is odd or even.

\smallskip
(a) Assume that $n$ is odd. 

Let $\mu_1 = (2,1)$ and $\mu_2 = (5,4,3,2,1)$ be $2$-core partitions of sizes $3$ and $15$, respectively.

We construct irreducible characters of $\Sym_n$ labelled by partitions with $2$-core either $\mu_1$ or $\mu_2$, and show that they contribute to distinct $2$-singular irreducible characters of $\Sym_n$.

For each $k$ with $0 \leq k \leq  \lfloor ({n-3})/{4}   \rfloor - 1$, define
\[
\alpha_k := (n - 1 - 2k, 1^{2k + 1}).
\]
Each $\alpha_k$ is a partition of $n$, with $2$-core $\mu_1$ and $2$-weight equal to $k + 1$. These are not self-conjugate and are pairwise non-conjugate. Their conjugates do not appear among the $\alpha_k$'s.
Similarly, for $1 \leq k \leq  \lfloor ({n - 2})/{4}   \rfloor$, define
$
\beta_k := (n - 1 - 2k, 2k + 1).
$
Again, these are partitions of $n$ with $2$-core $\mu_1$, and they are distinct and non-conjugate to one another or to any $\alpha_k$.

Let $\mu_2 = (5,4,3,2,1)$ be a $2$-core partition of size $15$.
Define
\[
\gamma_l := (n - 10 - 2l, 4,3,2,1^{2l+1}), \text{ for } 0 \leq l \leq  \lfloor ({n - 15})/{4}   \rfloor
\]
and 
$$
\delta_l := (n - 10 - 2l, 2l + 4, 3, 2, 1), \text{ for } 1 \leq l \leq  \lfloor ({n - 14})/{4}   \rfloor.
$$
These are again pairwise distinct, non-self-conjugate, and non-conjugate partitions of $n$ with $2$-core $\mu_2$.

Each partition constructed above corresponds to one irreducible $2$-singular character of $\Sym_n$ such that its restriction to $\Alt_n$ is $2$-singular. So we have:
\[
\begin{aligned}
n_2^*(G) &\geq  \lfloor \frac{n-3}{4}   \rfloor +  \lfloor \frac{n-2}{4}   \rfloor +  \lfloor \frac{n-15}{4}   \rfloor + 1 +  \lfloor \frac{n-14}{4}   \rfloor \\
&\geq n-\frac{34}{4}-4+1\ge n - 12=\frac{n}{2}+\frac{n}{2}-12\ge \frac{n}{2}.
\end{aligned}
\]
Define $t := \log_2(n)\ge 4$. Then we can verify that $t\leq 2^{t-1}=n/2$. Thus

\[
\de(P) \leq \log_2(n) \leq \frac{n}{2}  \leq n_2^*(G),
\]
as required.

\smallskip
(b) Assume $n \geq 24$ is even.
Let $\mu_1 = (3,2,1)$ and $\mu_2 = (4,3,2,1)$ be $2$-core partitions of $6$ and $10$, respectively.
Define
\[
\alpha_k := (n - 3 - 2k, 2, 1^{2k + 1}), \quad \text{for } 0 \leq k \leq  \lfloor ({n - 6})/{4}   \rfloor,
\]
and
\[
\beta_k := (n - 3 - 2k, 2k + 2, 1), \quad \text{for } 1 \leq k \leq \lfloor ({n - 5})/{4} \rfloor.
\]
These are all pairwise distinct, non-conjugate partitions of $n$ with $2$-core $\mu_1$.
Similarly, define
\[
\gamma_l := (n - 6 - 2l, 3, 2, 1^{2l + 1}), \quad \text{for } 0 \leq l \leq \lfloor ({n - 10})/{4} \rfloor,
\]
and
\[
\delta_l := (n - 6 - 2l, 2l + 3, 2, 1), \quad \text{for } 1 \leq l \leq \lfloor ({n - 9})/{4}\rfloor.
\]

Again, these partitions are distinct and have $2$-core $\mu_2$.
We have
\[
n_2^*(G) \geq  \lfloor \frac{n - 6}{4}   \rfloor + 1 + \lfloor \frac{n - 5}{4} \rfloor +  \lfloor \frac{n - 10}{4}   \rfloor + 1 +  \lfloor \frac{n - 9}{4}   \rfloor 
\geq n - 10\geq \frac{n}{2}.
\]
As above, we have
$
\de(P) \leq \log_2(n) \leq {n}/{2},
$
for all $n \geq 24$.
Hence,
$\de(P) \leq n_2^*(G),$
as desired.

In all cases, we have obtained $ \de(P)\leq  n_p^*(G)$, and the proof is complete.
\end{proof}

\subsection{Exceptional groups of Lie type}

In the remaining subsections of this section, we consider the finite simple groups of Lie type.
In this subsection, let $G$ be a finite simple group of Lie type defined over a finite field of size $q=q_0^f$, where $q_0$ is a prime. There exists a simple algebraic group $\mathbf{G}$ of adjoint type over the algebraic closure of a finite field $\mathbb{F}_q$, and a Steinberg endomorphism $F: \mathbf{G}\rightarrow \mathbf{G}$, such that $G=[\mathbf{G}^F,\mathbf{G}^F]$. Let $W$ denote the Weyl  group of $\mathbf{G}$ with respect to some $F$-stable maximal torus $\mathbf{T}$ of $\mathbf{G}$.
Note that the order of the Weyl group $W$ can be read off from \cite[p. 43]{Carter2}.

For a positive integer $n$, we denote by $\Phi_n$ the $n$th cyclotomic polynomial evaluated at $q$, that is, $\Phi_n=\Phi_n(q)$. For example, $\Phi_1=q-1$, $\Phi_2=q+1$ and so on.

\begin{lem}\label{lem:nonabelian}
Assume that the Sylow $p$-subgroups of $G$ are nonabelian for some prime divisor $p$ of $|G|$.  Then either $p\mid q$ and $\mathbf{G}$ is not of type $A_1$, or $p$ divides the order of $W$.
\end{lem}

\begin{proof}
This follows from Lemma 3 in \cite{MMN}.
\end{proof}

We now consider the finite simple exceptional groups of Lie type.

\begin{lem}\label{lem:exceptional} 
Let $G$ be a finite simple exceptional group of Lie type defined over a field of characteristic $q_0$. Let $p$ be a prime divisor of $|G|$ and let $P$ be a Sylow $p$-subgroup of $G$. Then $\de(P)\leq n_p^*(G).$ 
\end{lem}

\begin{proof}

As in Lemma \ref{lem:Alternating}, we may assume $P$ is nonabelian, so $|P|\ge p^3.$

\medskip
(1) $G={}^2{\rm B}_2(q^2)$, where $q^2=2^{2m+1},m\ge 1.$ From \cite{Suzuki}, $G$ has order $q^4(q^2-1)(q^4+1)$ and only  Sylow $2$-subgroups of  $G$ are nonabelian. These  have derived length $2$. By \cite[13.9]{Carter}, $G$ has unipotent characters of degree $q^4$ and $\sqrt{2}q\Phi_1\Phi_2/2$, which are both even. Hence $$n_2^*(G)\ge|\cd_2(G)|\ge 2=\de(P).$$

\medskip
(2) $G={}^2{\rm G}_2(q^2)$, where $q^2=3^{2m+1},m\ge 1.$ From \cite{Ward} for example, $|G|=q^6(q^2-1)(q^6+1)$ and only Sylow $3$-subgroups of $G$ are nonabelian; these have derived length $2$. By \cite[13.9]{Carter}, $G$ has unipotent characters of degree \[q^6,\frac{1}{\sqrt{3}}q\Phi_1\Phi_2\Phi_4,\frac{1}{2\sqrt{3}}q\Phi_1\Phi_2\Phi_{12}',\frac{1}{2\sqrt{3}}q\Phi_1\Phi_2\Phi_{12}''\]
where $\Phi_{12}=\Phi_{12}'\Phi_{12}''$ and $\Phi_{12}'=q^2-\sqrt{3q}+1,\Phi_{12}''=q^2+\sqrt{3q}+1$. It follows that $$n_3^*(G)\ge |\cd_3(G)|\ge 4>\de(P).$$

\medskip
(3) $G={}^2{\rm F}_2(q^2)$, where $q^2=2^{2m+1},m\ge 1.$ By Lemma \ref{lem:nonabelian} and \cite[p. 43]{Carter2}, we only need to consider the primes $p=2$ or $3$. By Proposition 5.4.13 in \cite{KL}, $G$ embeds into $\PGL_{26}({\overline{\mathbb{F}}_2})$. By Lemma \ref{lem2}, we have $\de(P)\leq [\log_2(26)]+1=5$.

Assume first that $p=2$. By using only the degrees of unipotent characters of $G$ listed in \cite[13.9]{Carter}, we can check that $|\cd_2(G)|>5\ge \de(P)$. Now assume that $p=3$. We see that $3$ divides $\Phi_4=q^2+1$. By checking \cite[13.9]{Carter} again, we see that $G$ has more than $5 $ distinct degrees which are all divisible by $\Phi_4$ and the result follows.

\medskip
(4) $G={}^3{\rm D}_4(q)$, where $q^2=2^{2m+1},m\ge 1.$ By Lemma \ref{lem:nonabelian} and \cite{Carter2}, we only need to consider the primes $2,3$ and the characteristic $q_0$. By \cite[Prop. 5.4.13]{KL} and Lemma \ref{lem2}, $\de(P)\leq [\log_2(8)]+1=4.$

Assume $p=q_0$. By \cite[13.9]{Carter}, $G$ has $7$ unipotent characters with pairwise distinct degrees and they are all $p$-singular. The lemma follows.

Next, assume $p=2$ and $2\neq q_0. $ Then $q$ is odd and $p\mid \Phi_1$ and $p\mid\Phi_2.$ Again, among the unipotent characters of $G$, there are $4$ characters whose degrees are  $$q^3\Phi_2^2\Phi_{12}/2, q^3\Phi_2^2\Phi_{6}^2/2,q^3\Phi_1^2\Phi_{3}^2/2,q^3\Phi_1^2\Phi_{12}/2.$$ These degrees are all even and pairwise distinct. So $|\cd_2(G)|\ge 4\ge \de(P)$.

Assume $p=3$ and $3\nmid q$. First assume that $q\equiv 1$ mod $3$. Then $3\mid q-1=\Phi_1$. We can use \cite[Table 4.4]{DM} to find at least $4$ irreducible characters of $G$ whose degrees are divisible by $q-1$ and pairwise distinct. Thus $\cd_3(G)\ge 4\ge \de(P)$. The case $q\equiv -2$ mod $3$ can be argued similarly.

\medskip
(5) $G={\rm G}_2(q)$ with $q\ge 3.$  By \cite[Prop. 5.4.13]{KL},  $G$ embeds into $\PGL_6(\overline{\mathbb{F}}_2)$ if $q$ is even and into $\PGL_7(\overline{\mathbb{F}}_q)$ if $q$ is odd. By Lemma \ref{lem2}, $\de(P)\leq [\log_2(7)]+1=3.$ By Lemma \ref{lem:nonabelian}, we only need to consider the primes $2,3$ and $q_0$.

If $3\leq q\leq 4$, we can use MAGMA \cite{magma} to verify that $|\cd_p(G)|\ge \de(P)$. So, we assume that $q\ge 5.$ Consequently, all the nontrivial unipotent characters of $G$ are $q_0$-singular and hence $|\cd_{q_0}(G)|\ge 7>\de(P)$ and the lemma holds in this case.

Assume $p=2$ and $2\nmid q.$ Then $G$ has three distinct unipotent characters of degree divisible by $\Phi_1$ and their degrees are pairwise distinct. Thus the result follows as well. Finally assume that $p=3$ and $3\nmid q.$  If $3\mid q-1$, then $G$ has three distinct unipotent characters of degree divisible by $\Phi_1$ and if $3\mid q+1=\Phi_2$, then $3$ divides $\Phi_6=q^2-q+1$. In this case, we can find 3 different unipotent characters of distinct degree and they are divisible by either $\Phi_2$ or $\Phi_6.$

\medskip
For the remaining families of simple groups, by \cite[Theorem 2.5]{Malle} every unipotent character of $G$ is $\Aut(G)$-invariant. Thus each $p$-singular unipotent character of $G$ forms an $\Aut(G)$-orbit in $\Irr_p(G)$ and it suffices to count such characters using \cite[13.9]{Carter}.

\medskip
(6) $G={\rm F}_4(q)$ with $q\ge 2.$ By Lemma \ref{lem:nonabelian} and \cite{Carter2}, we need to consider the primes $p=2,3$ and  the defining characteristic $q_0$.
From  \cite[Prop. 5.4.13]{KL} and Lemma \ref{lem2}, we have $\de(P)\leq [\log_2(26)]+1=5.$ 

If $p=q_0$, then by inspecting \cite[13.9]{Carter} there exist at least $5$ unipotent characters of $G$ whose degrees are divisible by $p$. Hence $n_p^*(G)\ge 5\ge\de(P)$. 

Assume $p=2$ and $2\nmid q.$ Then all the cyclotomic polynomials $\Phi_1,\Phi_2,\Phi_4$ and $\Phi_8$ are even. Referring to the list of unipotent characters of $G$ in \cite{Carter} again, we can find at least $5$ distinct unipotent characters whose degrees are divisible by $2$. Thus $n_2^*(G)\ge \de(P)$. 

Assume $p=3$ and $3\nmid q.$ Then either $3\mid (q-1)$, so  $\Phi_1,\Phi_3$ are divisible by $3$; or if $3\mid (q+1)$, so  $\Phi_2, \Phi_6$ are divisible by $3$. In either case, from  \cite[13.9]{Carter}, at least five unipotent characters have degrees divisible by $3$ and the lemma follows.

\medskip
(7) $G={\rm E}_6(q)$ or ${}^2{\rm E}_6(q)$ with $q\ge 2.$ By Lemma \ref{lem:nonabelian} and \cite{Carter2}, we consider the primes $p=2,3,5$ and the defining characteristic $q_0$.
By \cite[Prop. 5.4.13]{KL} and Lemma \ref{lem2}, $\de(P)\leq [\log_2(27)]+1=5.$ 
We provide  a detailed argument for $G={\rm E}_6(q)$; the  case ${}^2{\rm E}_6(q)$ is similar.

First, if $p=q_0$, then all $29$ nontrivial unipotent characters of  $G$ are $p$-singular. Hence  $n_p^*(G)\ge 29>5\ge \de(P)$.

If $p=2$ and $2\nmid q$, then $\Phi_1,\Phi_2\Phi_4$ and $\Phi_8$ are divisible by $2$.  From \cite[13.9]{Carter}, we can find more than five unipotent characters whose degrees are divisible by one of these cyclotomic polynomials and hence are even.  The claim now follows.

If $p=3$ and $3\nmid q$, then  either  $\Phi_1,\Phi_3$ or  $\Phi_2,\Phi_6$ are divisible by $3$. Again \cite[13.9]{Carter} provides at  least $5$ unipotent characters whose degrees are divisible by $3$ as wanted.

If $p=5$ and $5\nmid q$, then by Fermat's Little theorem $5\mid q^4-1$. Since $q^4-1=\Phi_1\Phi_2\Phi_4$, $p$ divides $\Phi_1, \Phi_2$ or $\Phi_4.$ In each possibility, we can find at least $5$ unipotent characters of $G$ whose degree is divisible by this cyclotomic polynomial. 

\medskip
(8) $G={\rm E}_7(q)$ with $q\ge 2.$ By Lemma \ref{lem:nonabelian} and \cite{Carter2}, the relevant  primes are $2,3,5,7$ and $q_0$. From \cite[Prop. 5.4.13]{KL} and Lemma \ref{lem2}, $\de(P)\leq [\log_2(56)]+1=6.$ 

The argument is similar to the previous case, as an illustration, consider the case $p=7$. Note that we can assume $7\nmid q$. Thus $7$ divides $q^6-1=\Phi_1\Phi_2\Phi_3\Phi_6$. From \cite[13.9]{Carter}, each $\Phi_i$, for each $i\in \{1,2,3,6\}$, divides the degree of at least $6$  unipotent characters. Therefore $n_7^*(G)\ge 6\ge \de(P).$

\medskip
(9) $G={\rm E}_8(q)$ with $q\ge 2.$ As above, we consider $p=2,3,5,7$ and $q_0$.
By \cite[Prop. 5.4.13]{KL} and Lemma \ref{lem2}, $\de(P)\leq [\log_2(248)]+1=8.$  The argument is analogous to the previous cases and we will skip the proof.

Therefore, in all cases above, we obtain $\de(P) \leq n_p^*(G)$, completing the proof of the lemma.
\end{proof}

\subsection{Classical groups in non-defining characteristic}\label{s:classical}

We are left to consider the finite classical groups. We now write $G=\mathbf{G}^F$ for a group of Lie type, where $\bG$ is a connected reductive algebraic group and $F\colon \bG\rightarrow\bG$ is a Frobenius morphism defining an $\FF_q$-rational structure on $\bG$, where $q$ is a power of a prime $q_0$. We let $(\bG^\ast, F)$ be dual to $(\bG, F)$. The characters $\irr{G}$ are then partitioned into so-called Lusztig series $\mathcal{E}(G,s)$, indexed by $G^\ast$-conjugacy classes of semisimple elements $s$ of the dual group $G^\ast:=(\bG^\ast)^F$. The set $\mathcal{E}(G,s)$ is  in bijection with the set $\mathcal{E}(\cent_{G^\ast}(s), 1)$ of unipotent characters of $\cent_{G^\ast}(s)$, under what is known as a Jordan decomposition of characters. Here if $\chi\in\mathcal{E}(G,s)$ is mapped to $\psi\in\mathcal{E}(\cent_{G^\ast}(s), 1)$ under Jordan decomposition, then we have $\chi(1)=[G^\ast:\cent_{G^\ast}(s)]_{q_0'}\psi(1)$. (See e.g. \cite[Sec.~2.6]{GM20} for a discussion.)

We will also often use freely the following several facts. From \cite[Lem. 4.4]{NT}, if $s\in [G^\ast, G^\ast]$ and $|\zent(G)|=|G^\ast|/|[G^\ast, G^\ast]|$, then the characters in $\mathcal{E}(G,s)$ are trivial on $\zent(G)$. Dual to this, we have the linear characters of $G$ are in bijection with elements of $\zent(G^\ast)$, and $\mathcal{E}(G, sz)=\mathcal{E}(G, s)\otimes \hat{z}$, where $\hat{z}\in\Irr(G/[G,G])$ corresponds to $z\in \zent(G^\ast)$. (See e.g. \cite[Prop. ~11.4.12, Rem.~11.4.14]{DM20}).  Further, for $\varphi\in\Aut(G)$, $\mathcal{E}(G, s)$ is mapped under $\varphi$ to $\mathcal{E}(G, s^{\varphi^\ast})$ where $\varphi^\ast\in\Aut(G^\ast)$ is some dual automorphism  (see \cite[Cor.~2.4]{NTT}). 
Further, recall that unipotent characters are independent of isogeny type, so that they are trivial on $\zent(G)$ and restrict irreducibly to $[G, G]$. 

\begin{lem}\label{lem:GLunip}
 Let $G=\GL_{p^k}(\epsilon q)$ with $p^k\geq 5$ and $p\mid (q-\epsilon)$ a prime. Then $G$ has at least $k+1$ unipotent characters of distinct degrees divisible by $p$.
 \end{lem}
\begin{proof}
The unipotent characters of $G$ are indexed by partitions $\lambda=(\lambda_1, \lambda_2,\ldots,\lambda_r)$ of $n$, where $\lambda_1\geq \lambda_2\geq\cdots\geq \lambda_r$. The degree of the character $\chi_\la$ corresponding to $\la$ is given by a hook formula: 
\[\chi_\lambda(1)=q^{n(\la)}\prod_{i=1}^n\frac{q^i-(\epsilon)^i}{q^{l_i}-(\eps)^{l_i}}\] where $l_i$ is the length of the hook at the $i$th box of the corresponding Young diagram and $n(\la):=\sum_{i=1}^r(i-1)\lambda_i$. (See \cite[Props.~4.3.1, 4.3.5]{GM20}.) We also recall that $q^m-1$ is a product of the cyclotomic polynomials $\Phi_d:=\Phi_d(q)$ such that $d\mid m$; that $\Phi_d$ is divisible by $p$ if and only if $d=d_p(q)p^t$ for some $t\geq 0$ (where $d_p(q)$ is the order of $q$ modulo $p$ if $p$ is odd, resp. the order of $q$ modulo $4$ if $p=2$); and that $p\mid\mid \Phi_d$ if $t>0$.

If $p$ is odd, consider the partitions $(p^k-p^j-1, p^j+1)$ for $0\leq j<k$ and $(p^k-3, 2, 1)$. From the degree formula and discussion of divisibility for $\Phi_d$ above, we see these each have degree divisible by $p$. Further, the degrees have distinct powers of $q$ dividing them, except possibly the case $p=3$, where the degrees of the characters corresponding to $(p^k-4, 4)$ and $(p^k-3, 2, 1)$ have the same power of $q$. In that case, however, we can see that the $q_0'$-part of the degrees are different.

If $p=2$ (so $p^k\geq 8$), consider the partitions $( 2^k-2^j-1, 2^j+1)$ for $0\leq j<k-1$, $(2^k-3, 2, 1)$, and $(2^k-4, 2, 1, 1)$. We similarly note that the $q_0$-parts of the degrees are distinct, and that these degrees are all even.
\end{proof}

\begin{rem}\label{rem:GLevenunip}
 When $p=2$, to make up for the condition $2^k\geq 8$ in Lemma \ref{lem:GLunip} (note that $\GL_4(\eps q)$ has only one unipotent character of even degree), in what follows it will be useful to note that if $G=\GL_{n}(\epsilon q)$ with $5\leq n\leq 7$, $\Sp_6(q)$, $\SO_{7}(q)$, or $\SO^+_8(q)$ and $q$ is a power of an odd prime, then $G$ has at least $3$ unipotent characters of distinct, even degrees. In fact, in the latter three cases, there are at least $4$ unipotent characters of distinct even degrees. This can be seen directly from the list of unipotent character degrees.
 \end{rem}
 
\begin{prop}\label{prop:classical}
Let $G$ be $\GL_n(\epsilon q)$ with $n\geq 2$, $\Sp_{2n}(q)$ with $n\geq 2$, $\SO_{2n+1}(q)$ with $n\geq 3$, or $\SO^\pm_{2n}(q)$ with $n\geq 4$, and let $p$  be an odd  prime not dividing $q$. Let $Q\in\Syl_p(G)$ and assume that $Q$ is nonabelian. Then $G$ has at least $\mathrm{dl}(Q)$ characters in $\Irr_p(G)$ with distinct degrees that are trivial on $\zent(G)$. In the case $G=\GL_n(\epsilon q)$, these further restrict irreducibly to $\SL_n(\epsilon q)$. 
\end{prop}
\begin{proof}
 
Recall that in the orthogonal and symplectic cases, there is some $m$ such that $Q$ is isomorphic to a Sylow $p$-subgroup of $\GL_m(q)$.  In the case $\GL_n(\epsilon q)$, let $m:=n$, so that in all cases,  $m$ is such that $Q$ is a Sylow $p$-subgroup of some $\GL_m(\epsilon q)$. Let $d':=d_p(q)$ unless $G=\GL_n(-q)$, in which case we let $d':=d_p(-q)$. 

Write $m=d'w+r$ with $0\leq r<d'$,  and let $w=\sum_{i=0}^k a_ip^i$ be the $p$-adic expansion of $w$ as before, where $a_k\neq 0$. Note that $k\neq 0$ by our assumption that $Q$ is nonabelian.  Let $\zeta\in\FF_{q^{2d'}}^\times$ have order $p$. Then there is a semisimple element $s'$ of $\GL_m(\epsilon q)$ whose nontrivial eigenvalues are the orbit of $\zeta$ under $x\mapsto x^{\epsilon q}$, each with multiplicity $p^k$. This in turn yields a semisimple element $s$ of 
$G^\ast\cong \GL_{n}(\epsilon q), \SO_{2n+1}(q)$, $\Sp_{2n}(q)$,  resp. $\SO^\pm_{2n}(q)$ whose  centralizer is of the form $\cent_{G^\ast}(s)\cong \GL_{p^k}(\eta q^{e})\times X$, where $X$ is some classical group, $e\in\{d', d'/2\}$, $\eta\in\{\pm1\}$, and $p\mid (q^{e}-\eta)$. (Here $e=d'$ if $G=\GL_n(\epsilon q)$ and $e=d_p(q^2)$ in the other cases.)

Here, note that $s$ lies in $[G^\ast, G^\ast]$, so the characters in $\mathcal{E}(G,s)$ are trivial on the center. Further, in the case $G=\GL_n(\epsilon q)$, we see from the multiplicities of eigenvalues that $s$ is not $G^\ast$-conjugate to any $sz$ for $1\neq z\in\zent(G^\ast)$. From this we see that the characters in $\mathcal{E}(G, s)$ must restrict irreducibly to $\SL_n(\epsilon q)$. 

First assume that $p^k\geq 5$. Then  by Lemma \ref{lem:GLunip}, we have $\mathcal{E}(\cent_{G^\ast}(s), 1)$ contains at least $k+1$ characters of distinct degrees divisible by $p$ of the form $\psi\otimes 1_X$ with $\psi\in\Irr_p(\GL_{p^k}(\eta q^{e}))$. By the degree properties of Jordan decomposition, this yields at least $k+1=\mathrm{dl}(Q)$ characters of $\mathcal{E}(G,s)\cap \Irr_p(G)$ of distinct degrees, completing the statement in this case.

Now let $p^k=3$. Then $\mathrm{dl}(Q)=2$, $w=3a_1+a_0$ and $d'\in\{1,2\}$. If $G$ is one of $\GL_3(\epsilon q)$, $\GL_4(\eps q)$, $\GL_5(\eps q)$, $\GL_6(\epsilon q)$, $\Sp_4(q)$, $\Sp_6(q)$, $\SO_5(q)$, $\SO_7(q)$, or $\SO_{8}^\pm(q)$, the statement can be checked directly using the known character tables and/or unipotent character degrees. In particular, $\GL_6(\eps q)$ has at least two distinct unipotent character degrees divisible by $3$. For relevant larger values of $n$, we may then argue just as above, taking $s'$ instead to have the orbit of $\zeta$ as eigenvalues with multiplicity $6$.
\end{proof}

\begin{prop}\label{prop:classical2}
Let $q$ be odd and let $G$ be $\GL_n(\epsilon q)$ with $n\geq 2$, $\Sp_{2n}(q)$ with $n\geq 2$, $\SO_{2n+1}(q)$ with $n\geq 3$, or $\SO^\pm_{2n}(q)$ with $n\geq 4$, and let $p=2$. Let $Q\in\Syl_2(G)$. Then $G$ has at least $\mathrm{dl}(Q)$ characters $\chi_1,\ldots,\chi_{\mathrm{dl}(Q)}$ in $\Irr_2(G)$ that are trivial on the center and restrict irreducibly to $[G,G]$. These further satisfy that $\chi_i^\sigma\neq \chi_j\beta$ for any $i\neq j$, $\beta\in\Irr(G/[G,G])$, and $\sigma\in\Aut(G)$. 
\end{prop}
\begin{proof}
First assume that $G=\GL_n(\eps q)$. Recall that $\de(Q)=k+1$, where $n=2^k+a_{k-1}2^{k-1}+\cdots+ a_12+a_0$ is the $2$-adic expansion of $n$. From Lemma \ref{lem:GLunip} and Remark \ref{rem:GLevenunip}, we see that for $5\leq n\leq 8$ or for $n=2^k$ with $k\geq 3$, the statement holds. For $2\leq n\leq 3$, we see from the known character tables of $G$ and $[G, G]$ that $G$ has at least $2$ even-degree characters of different degrees that satisfy the statement. For $\GL_4(\eps q)$, we have one unipotent character of even degree, and considering a semisimple element $s\in G^\ast$ with eigenvalue $\zeta_4, \zeta_4^{-1}, 1, 1$, where $\zeta_4$ is a primitive $4$th root of unity, we have $\cent_{G^\ast}(s)\cong \GL_2(\eps q)\times X$, where $X$ is abelian. Then the two characters in $\mathcal{E}(G,s)$ have distinct even degrees distinct from that of the unipotent character of even degree. Further, these are trivial on the center since $s\in[G^\ast, G^\ast]$ and restrict irreducibly to $\SL_4(\eps q)$ since we see from the eigenvalues that $sz$ is not conjugate to $s$ for any $1\neq z\in\zent(G^\ast)$.

So, assume that $G=\GL_n(\eps q)$ with $n\geq 9$ not a power of $2$. We argue similarly to Proposition \ref{prop:classical}. Let $s$ be a semisimple element of $G^\ast\cong G$ whose only nontrivial eigenvalues are $-1$ with  multiplicity $2^k$. Then $s\in[G^\ast, G^\ast]$, so the characters in $\mathcal{E}(G, s)$ are trivial on $\zent(G)$. Further,   $\cent_{G^\ast}(s)\cong \GL_{2^{k}}(\eps q)\times X$ for a lower-rank linear or unitary group $X$, so $\mathcal{E}(\cent_{G^\ast}(s), 1)$ contains at least $k+1$ even-degree characters with distinct degree by Lemma \ref{lem:GLunip}. Hence the same is true for $\mathcal{E}(G,s)$ by the degree properties of Jordan decomposition. Further, from the eigenvalues we see that $s$ is not $G^\ast$-conjugate to $sz$ for any $1\neq z\in\zent(G^\ast)$, meaning that these must restrict irreducibly to $\SL_n(\eps q)$.

Now let $G$ be one of the symplectic or special orthogonal groups listed, and recall that $\de(Q)\leq k+1$ where $2n=2^k+\cdots+a_12+a_0$. 
First let $G\in\{\Sp_{2n}(q), \SO_{2n+1}(q)\}$. If $n=2$, then $G=\Sp_4(q)$ and we can check the statement from the known character table \cite{srinivasan}. If $G=\Sp_6(q)$ or $\SO_7(q)$, then as noted in Remark \ref{rem:GLevenunip}, there are more than 3 unipotent characters of distinct even degrees. If $n>3$, then there is a semisimple element $s\in [G^\ast, G^\ast]$ whose eigenvalues are $\pm1$ and with centralizer of the form  $X_1\times X_2$ with $X_1\in \{\Sp_6(q), \SO_7(q)\}$ and $X_2$ an appropriate classical group. Recall from Remark \ref{rem:GLevenunip} that there are then four distinct degrees of unipotent characters in $\Irr_2(X_1)$. We note that two of these satisfy $\chi(1)_q=q$, and two satisfy $\chi(1)_q=q^4$. Then taking the eight unipotent characters obtained by tensoring these with $1_{X_2}$ and $\mathrm{St}_{X_2}$ yields eight distinct unipotent character degrees in $\mathcal{E}(\cent_{G^\ast}(s), 1)$, and hence eight distinct even character degrees in $\mathcal{E}(G,s)$. Further, by comparison of eigenvalues, we see $sz$ is not conjugate to $s$ for $1\neq z\in\zent(G^\ast)$ so these characters are irreducible on restriction to $[G, G]$, except possibly if $G^\ast=\Sp_{12}(q)$. In the latter case, we can instead consider $s$ to have eigenvalues $\pm1$ but with $X_1\in\{\Sp_4(q), \SO_5(q)\}$. Then $sz$ is not conjugate to $s$ for $1\neq z\in\zent(G^\ast)$, $s\in [G^\ast, G^\ast]$, and $\Irr_2(X_1)$ contains two unipotent characters with distinct degrees;  the tensors with $1_{X_2}$ and $\mathrm{St}_{X_2}$ give 4 distinct even-degree unipotent characters in $\cent_{G^\ast}(s)$, and hence four distinct even degree characters in $\Irr_2(G)\cap\mathcal{E}(G,s)$ trivial on the center and irreducible on restriction to $[G,G]$.

Hence we may assume that $k>7$.
Let $s_1, s_2\in G^\ast\in\{\SO_{2n+1}(q), \Sp_{2n}(q)\}$ have nontrivial eigenvalues $\zeta_4$ and $-\zeta_4$, each with multiplicity $2^{k-2}$ in $s_1$ and multiplicity $2^{k-3}$ in $s_2$. Then $\cent_{G^\ast}(s_1)\cong \GL_{2^{k-2}}(\eta q) \times X_1$ and $\cent_{G^\ast}(s_2) \cong \GL_{2^{k-3}}(\eta q) \times X_2$ for some appropriate classical groups $X_i$ and $\eta\in\{\pm1\}$. Since $k-3> 3$, we have by Lemma \ref{lem:GLunip} at least $k-1$ characters of distinct even degree in $\mathcal{E}(\cent_{G^\ast}(s_1), 1)$ and at least $k-2$ characters of distinct even degree in $\mathcal{E}(\cent_{G^\ast}(s_2), 1)$, yielding at least $k-1$, resp. $k-2$ characters of distinct even degree in $\Irr_2(G)\cap \mathcal{E}(G, s_1)$, resp. $\Irr_2(G) \cap\mathcal{E}(G, s_2)$ by degree properties of Jordan decomposition. 
Considering the eigenvalue multiplicities, we see $s_i$ cannot be conjugate to $s_iz$ for $1\neq z\in\zent(G^\ast)$, so the characters restrict irreducibly to $[G,G]$. Similarly, we see $s_1^\sigma$ cannot be $G^\ast$-conjugate to $s_2z$ for $ z\in\zent(G^\ast)$ and $\sigma\in \Aut(G^\ast)$, yielding that no two of these $k-1+k-2>k+1$ characters  $\chi, \chi'$ can satisfy $\chi^\sigma=\chi'\beta$ for $\beta\in\Irr(G/[G, G])$ and $\sigma\in\Aut(G)$. We also note that as elements of $\GL_{n}(\eta q)$ embedded into $G^\ast$, both $s_1$ and $s_2$ have determinant 1, so that each are in $[G^\ast, G^\ast]$, and hence all of the characters are trivial on $\zent(G)$.

We finally consider $G=\SO_{2n}^\pm(q)$ with $2n=2^k+\cdots +a_0$ and note that $G^\ast\cong G$ here.  We remark that $\SO_8^+(q)$ (or any Lie-type group of type $\type{D}_4(q)$) has at least 6 distinct even unipotent degrees. On the other hand, $\SO_8^-(q)$ has two distinct even unipotent degrees. In this case, we may consider a semisimple element $s\in \SO_8^-(q)\cong G^\ast$ with eigenvalues $\pm1$ and centralizer of type $\type{D}_3$ or $\tw{2}\type{D}_3$, which yields an additional   even-degree character trivial on the center that restricts irreducibly. Since $G^\ast$ has a self-normalizing Sylow $2$-subgroup, any semisimple $s\in G^\ast$ of odd order will also yield a semisimple character of $G$ with even degree trivial on the center and irreducible on restriction. Further, we again see no two of these four characters can satisfy $\chi^\sigma=\chi'\beta$ for $\beta\in\Irr(G/[G, G])$ and $\sigma\in\Aut(G)$. 

Now assume $n>4$. Then $G^\ast$ has a semisimple element $s$ lying in $[G^\ast, G^\ast]$ with $-1$ appearing as an eigenvalue with multiplicity 8, whose centralizer has type $\type{D}_4.\type{D}_{n-4}^\pm$. Then we see $\mathcal{E}(G,s)$ has at least $6$ characters of distinct even degrees, and these are trivial on the center.  Further, considering the eigenvalues again yields $sz$  is not conjugate to $s$ for a nontrivial $z\in \zent(G^\ast)$, except possibly when $n=8$. In this case we can instead consider $s$ with eigenvalues $\pm1$ and centralizer type $\type{D}_3.\type{D}_5^{\pm}$. Groups of type $\type{D}_3=\type{A}_3$ have one unipotent character of even degree, and those of type $\type{D}_5^\pm$ have at least 5 unipotent characters of distinct degrees. Then their are at least five characters in $\mathcal{E}(\cent_{G^\ast}(s), 1)$ with distinct even degrees, and hence again at least five characters in $\Irr_2(G)\cap \mathcal{E}(G,s)$ with distinct degrees. As before, these are again trivial on the center and restrict irreducibly to $[G,G]$. 

Then we are left to consider the case $k\geq 6$. Here we consider $s_1$ and $s_2$ as before, with $\zeta_4$ and $\zeta^{-1}$ each having multiplicities $2^{k-2}$, resp. $2^{k-3}$. The same argument as before yields $k-1+k-2> k+1$ characters of even degree with the desired properties.
\end{proof}
\begin{cor}\label{cor:classicalsimple}
Let $S$ be a simple group such that $S=\operatorname{PSL}_n(\epsilon q)$ with $n\geq 2$, $\operatorname{P\Omega}_{2n+1}(q)$ with $n\geq 3$, $\operatorname{PSp}_{2n}(q)$ with $n\geq 2$, or $\operatorname{P\Omega}_{2n}^\pm(q)$ with $n\geq 4$. Let $p$  be a prime not dividing $q$. Then $\mathrm{dl}(P)\leq n_p^\ast(S)$, where $P\in\Syl_p(S)$.
\end{cor}
\begin{proof}
    We may write $S=H/\zent(H)$ for some $H\leq G$ with $G$ as in Propositions \ref{prop:classical} or \ref{prop:classical2}. Let $Q\in\Syl_p(G)$. If $Q$ is abelian, then so is $P$ and the statement follows from the It{\^o}--Michler theorem as before.

    Then we assume $Q$ is not abelian.     
    By Propositions \ref{prop:classical} and \ref{prop:classical2}, we have $\mathrm{dl}(Q)$ characters $\wt\chi_1, \ldots, \wt\chi_{\mathrm{dl}(Q)}\in\Irr_p(G/\zent(G))$ such that $\wt\chi_i^\sigma\neq \wt\chi_j\beta$ for any $\sigma\in\Aut(G)$ and $\beta\in\Irr(G/[G,G])$. Then the characters $\wt\chi_i$ must yield at least $\mathrm{dl}(Q)\geq \de(P)$ characters $\chi_1, \ldots, \chi_{\mathrm{dl}(Q)}\in\Irr(S)$ that must come from distinct $\Aut(S)$-orbits. It suffices now to argue that these constituents can be chosen to have degree divisible by $p$.

    If $p$ is odd and $S$ is not $\operatorname{PSL}_n(\epsilon q)$, then we have $[G:H]\leq 2$, $|\zent(H)|= |\zent(G)|\leq 2$, and certainly the $\chi_i$ have degree divisible by $p$.
If  $p$ is odd and $S=\operatorname{PSL}_n(\epsilon q)$, then $H=\SL_n(\epsilon q)$ and we note that $\chi_i=\wt\chi_i|_S\in\Irr_p(S)$ for $1\leq i\leq \mathrm{dl}(Q)$, as guaranteed by Proposition \ref{prop:classical}. Similarly, if $p=2$, Proposition \ref{prop:classical2} guarantees the $\wt\chi_i$ are again irreducible when restricted to $S$, completing the proof.
\end{proof}

\subsection{Classical groups in defining characteristic}

We finally consider the case of classical groups whose underlying characteristic is the same as $p$.
\begin{lem}\label{lem:classicaldefining}
Let $q$ be a power of the prime $p$ and let $S$ be a simple group such that  $S=\operatorname{PSL}_n(\epsilon q)$ with $n\geq 2$, $\operatorname{P\Omega}_{2n+1}(q)$ with $n\geq 3$, $\operatorname{PSp}_{2n}(q)$ with $n\geq 2$, or $\operatorname{P\Omega}_{2n}^\pm(q)$ with $n\geq 4$. Then $n_p^*(S)\ge \de(P)$, where $P\in\Syl_p(S)$. 
\end{lem}

\begin{proof}
As before, we may write $S=H/\zent(H)$ for some $H\leq G$ with $G$ one of the classical groups studied above.  Then we may view $P$ as a quotient of a subgroup of a Sylow $p$-subgroup of $\GL_n(\epsilon q)$, resp. $\GL_{2n+1}(q), \GL_{2n}(q), \GL_{2n}(q)$. Note then that from Lemma \ref{lem3}, we obtain a bound $$\de(P)\leq \lfloor \log_2(2n)\rfloor+1 \leq \log_2(n)+2,$$ which can be improved to $\de(P)\leq  \log_2(n)+1$ when $S=\operatorname{PSL}_n(\epsilon q)$.

For $n\leq 5$, we may use the list of character degrees for the corresponding adjoint group (which is then the group $\mathrm{InnDiag}(S)$ of inner-diagonal automorphisms of $S$) available at \cite{Luebeck} and the bound above for $\de(P)$ to check directly that $|\mathrm{cd}_p(\mathrm{InnDiag}(S))|\geq \de(P)$, except possibly for the case $\operatorname{PSp}_4(2)$, which is not simple and whose derived subgroup $\Alt_6$ has already been considered. Since $\mathrm{InnDiag}(S)/S$ is an abelian group with size prime to $p$, this shows $n_p^\ast(S)\geq \de(P)$ in these cases, as desired.
Hence, we may assume that $n\geq 6$.

For the moment, assume that  $S\not\in\{\operatorname{P\Omega}_{2n+1}(2), \operatorname{PSp}_{2n}(2)\}.$ Then by \cite[Thm.~6.8]{Ma07}, all nontrivial unipotent characters of $S$ are in $\Irr_p(S)$.
By \cite[Thm.~4.5.11]{GM20}, the unipotent characters lie in orbits of size at most two under $\Aut(S)$ for $\operatorname{P\Omega}_{2n}^+(q)$,
and orbits of size one otherwise. In all cases, the Steinberg character is $\Aut(S)$-invariant.

If $S=\operatorname{PSL}_n(\epsilon q)$, the number of unipotent characters is the number $p(n)$ of partitions of $n$, so we have $n_p^\ast(S)\geq p(n)-1$. In the remaining cases, there are at least $k(2,n)/2$ unipotent characters (where $k(2,n)$ is the class number of the reflection group $G(2,1,n)$), which is again at least $p(n)$. Hence it suffices to show that $({p(n)-2})/{2}+1={p(n)}/{2}\geq \log_2(n)+2$. 

By \cite[Corollary 3.1]{Maroti}, we know that \[p(n)>\frac{e^{2\sqrt{n}}}{14}.\]  It then suffices to know that  ${e^{2\sqrt{n}}}/{14}\ge 2\log_2(n)+4,$ which indeed holds for $n\geq 6$. 

Now assume that $S=\operatorname{PSp}_{2n}(2)$ (which is the same as $\operatorname{P\Omega}_{2n+1}(2)$), with $n\geq 6$ and $p=2$. Here there are four nontrivial unipotent characters with odd degree. Since in this case unipotent characters are $\Aut(S)$-invariant, it suffices here to show that $p(n)-5\geq \log_2(n)+2$. Since ${e^{2\sqrt{n}}}/{14}\geq \log_2(n)+7$ for $n\geq 7$, and the case $\operatorname{Sp_{12}}(2)$ can be checked directly in GAP, this completes the proof. 
\end{proof}

\medskip

\begin{proof}[Proof of Theorem \ref{th:simple groups}]
Theorem \ref{th:simple groups} now follows by combining Lemmas \ref{lem:sporadic},  \ref{lem:Alternating},  \ref{lem:exceptional}, and \ref{lem:classicaldefining} with Corollary \ref{cor:classicalsimple}.
\end{proof}

\begin{proof}[Proof of Theorem \ref{th:main}]
Theorem \ref{th:main} follows from Theorems \ref{th:simple groups} and \ref{mathcal C}.
\end{proof}

\section{Block version} \label{s:block}

We now prove Theorem \ref{maink}.  We begin with an easy lemma.  The following is part of the proof of Theorem 3.2 of \cite{brauergraph1}, but as the proof is short, we include it here.

\begin{lem}\label{opd}  Suppose $G$ is $p$-solvable and $N = \opd(G)$ is central in $G$.  Let $P = \opp(G)$.  Then $$NP'/P' = \opd(G/P').$$
\end{lem}

\begin{proof}  Note that $NP/N = \opp(G/N)$, so we write $M = NP = N \times P$.  Write $K = NP'$, so that clearly $$K/P' \leq \opd(G/P').$$ Let $L/P' = \opd(G/P')$ and assume to get a contradiction that $K < L$.  It follows that if $H \leq L$ is a $p$-complement, then $H > N$.

We have that $L/P'$ centralizes $P/P'$ and thus $H$ centralizes $P/P'$.  Let $C$ be the subgroup of elements of $P$ that are fixed by $H$.  As fixed points come from fixed points in coprime actions, we see that $P = CP'$.  As $P$ is nilpotent, we have that $P' \leq \Phi(P)$.  Thus $P = CP' \leq C \Phi(P) \leq P$, and thus $P = C\Phi(P)$.  It follows that $C = P$, so that $H$ fixes all of $P$.  However, this contradicts the Hall-Higman lemma (see for instance \cite{group}) applied to $G/N$.  This contradiction completes the proof.
\end{proof}

\begin{proof}[Proof of Theorem \ref{maink}]  We work by induction on $|G : \opd(G)| + |D|$.  Let $N = \opd(G)$ and suppose $\alpha \in \Irr(N)$ is covered by $B$.  If $\alpha$ is not invariant in $G$, we let $T = G_{\alpha} < G$.  By the Fong-Reynolds theorem (see for instance \cite{Navarro_b1}), there is a block $B_0$ of $T$ such that a defect group for $B_0$ is a defect group for $B$, and character induction is a bijection from $B_0$ to $B$.  By the inductive hypothesis, $\dlen(D) \leq n_p(B_0) + 1$.  As each character of positive height in $B_0$ induces to a distinct character of positive height in $B$, it then follows that $\dlen(D) \leq n_p(B) + 1$.

Thus we may assume that $\alpha$ is invariant in $G$, and thus it follows from the Fong reduction (see for instance Theorem 10.20 of \cite{Navarro_b1}) that $D$ is a full Sylow $p$-subgroup of $G$ and $B$ consists of all of the characters (both ordinary and Brauer) lying over $\alpha$.  Note that the characters of positive height in $B$ are now exactly the characters of $p$-singular degree in $B$.  By replacing the triple $(G, N, \alpha)$ with an isomorphic character triple, if necessary (which certainly preserves the defect group and the number of $p$-singular characters lying over $\alpha$) we may assume that $N = \opd(G)$ is central in $G$.  

Write $P = \opp(G)$ and note that if we set $M = N \times P$, then $M/N = \opp(G/N)$.  We see that $M = {\bf{F}}(G)$, and if we define $F_i = {\bf{F}}_i(G/M)$, then $F_1/M$ is a $p'$-group.

We now reduce to the case that $P$ is abelian.  Assume that $P$ is not abelian, and note that by Lemma \ref{opd} we have that $NP'/P' = \opd(G/P')$.  Let ${\widehat{\alpha}}$ denote the unique $p'$-special extension of $\alpha$ to $NP'$, and we may consider ${\widehat{\alpha}}$ as an invariant character of $NP'/P'$.  Thus we may consider the block $b'$ of $G/P'$ consisting of characters lying above ${\widehat{\alpha}}$, and note that every character in $b'$ is also in $B$.  Of course, $b'$ has defect group $D/P'$.  By the inductive hypothesis, $$\dlen(D/P') \leq n_p(b') + 1.$$   

By Taketa's theorem, we see that $\dlen(P) - 1 \leq \cd_p(P)$.  Suppose $\beta$ is a $p$-singular character of $P$.  Then $\alpha \times \beta \in \Irr(M)$ does not have $P'$ in the kernel, and thus there exists a character $\chi$ in $\Irr(B)$ lying above $\alpha \times \beta$ that is $p$-singular and does not have $P'$ in its kernel, and thus is not counted by $n_p(b')$. Thus we have $$n_p(B) \geq n_p(b') + \cd_p(P) \geq \dlen(D/P') - 1 + \dlen(P) - 1,$$ which in turn is equal to  $$\dlen(D/P') + \dlen(P') - 1 \geq \dlen(D) - 1,$$ and we are done if $P$ is not abelian.

Now let $K/P = \opd(G/P)$, so that $P \leq M \leq K$ and $M \leq F_1 \leq K$.  As $D$ is a full Sylow $p$-subgroup of $G$, there exists a Brauer character $\varphi \in \ibr(B)$ of $p'$-degree, and we let $\chi \in \Irr(B)$ be the $p'$-special lift (see \cite{isaacssolvablebook}) of $\varphi$, so that $\chi \in \Irr(G/P)$ and $\varphi \in \ibr(G/P)$.  Note that the defect group of the block $b$ of $G/P$ containing $\chi$ is a full Sylow $p$-subgroup of $G/P$.  Of course, any character in $b$ is also a character in $B$ (with $P$ in its kernel). By the inductive hypothesis, we see that $$\dlen(D/P) \leq n_p(b) + 1.$$  As $\dlen(D) \leq \dlen(D/P) + 1$, then it is enough to show that there is at least one $p$-singular character in $B$ that does not have $P$ in its kernel. 

Suppose, then, for every $1 \neq \lambda \in \Irr(P)$, we have that every constituent of $(\alpha \times \lambda)^G$ has $p'$-degree.  As there must be at least one such nontrivial character $\lambda \in \Irr(P)$, it follows by Theorem A of \cite{gluckwolf} that $G/M$ has an abelian Sylow $p$-subgroup, and thus we have $\dlen(D) \leq 2$ and we are done.
\end{proof}


\begin{thebibliography}{100}


\bibitem{magma} 
W. Bosma, J. Cannon, and C. Playoust, {The {\textsc{Magma}} algebra system I: The user language}, {\em J. Symb. Comput.} \textbf{24} (1997), 235--265.

\bibitem{Carter} R.~W. Carter, {\it Finite groups of Lie type}, Pure and Applied Mathematics (New York), John Wiley \& Sons, Inc., New York, 1985.

\bibitem{Carter2} R.~W. Carter, {\it Simple groups of Lie type}, Pure and Applied Mathematics, Vol. 28, John Wiley \& Sons, London, 1972.

\bibitem{carterfong} R.~W. Carter\ and\ P. Fong, The Sylow $2$-subgroups of the finite classical groups, {\em J. Algebra {\bf 1}} (1964), 139--151. 

\bibitem{brauergraph1}  J.~P. Cossey and M. L. Lewis, Bounding the diameter of the Brauer graph of a block of a solvable group, {\em Proc. Edinb. Math. Soc. {\bf 59}} (2016), no.~2, 591-603.

\bibitem{DM} D.~I. Deriziotis\ and\ G.~O. Michler, Character table and blocks of finite simple triality groups $^3D_4(q)$, {\em Trans. Amer. Math. Soc. {\bf 303}} (1987), no.~1, 39--70.

\bibitem{DM20} F. Digne and J.~C.~M. Michel, {\it Representations of finite groups of Lie type}, second edition, 
London Mathematical Society Student Texts, 95, Cambridge Univ. Press, Cambridge, 2020.

\bibitem{EaMo} C. W. Eaton, and A. Moret\'o, Extending Brauer's height zero conjecture to blocks with nonabelian defect groups, {\em Int. Math. Res. Not. IMRN, {\bf 20}} (2014), 5581--5601. 

\bibitem{FLZ1} Z. Feng, Y. Liu, J. Zhang, Towards the Eaton-Moret\'o conjecture on the minimal height of characters, {\em Comm. Algebra {\bf 47}} (2019), 5007-5019.

\bibitem{FLZ} Z. Feng, Y.~J. Liu and J.~P. Zhang, On heights of characters of finite groups, {\em J. Algebra {\bf 556}} (2020), 106--135.

\bibitem{GAP} The GAP Group, GAP  Groups, Algorithms, and Programming, \url{ http://www.gap-system.org.}

\bibitem{GM20}
{ M. Geck and G. Malle},
\newblock {\em The Character Theory of Finite Groups of {L}ie Type}, volume 187
of {Cambridge Studies in Advanced Mathematics}.
\newblock Cambridge University Press, Cambridge, 2020.

\bibitem{GMS}
E. Giannelli, J.~M. Mart\'inez and A.~A. Schaeffer~Fry, Character degrees in blocks and defect groups, {\em J. Algebra {\bf 594}} (2022), 170--193.

\bibitem{gluckwolf}  D.~Gluck and T. Wolf, Brauer's height conjecture for $p$-solvable groups, {\em Trans. Amer. Math. Soc. {\bf 282}} (1984), no. ~1, 137-152.

\bibitem{GO} A.~J. Granville\ and\ K. Ono, Defect zero $p$-blocks for finite simple groups, {\em Trans. Amer. Math. Soc. {\bf 348}} (1996), no.~1, 331--347. 
 
\bibitem{GGLMNT} D. Goldstein, R. M. Guralnick, M. L. Lewis, A. Moret\'o, G. Navarro, and P. H. Tiep, Groups with exactly one irreducible character of degree divisible by $p$, {\em  Algebra Number Theory} {\bf 8} (2014), no. 2, 397--428.

\bibitem{Huppert} B. Huppert, B. Huppert, {\it Endliche Gruppen. I}, Die Grundlehren der mathematischen Wissenschaften, Band 134, Springer, Berlin, 1967.  Translated. Springer Cham. 2025.


\bibitem{Isaacs} I. M. Isaacs, \emph{Character theory of finite groups}. Academic Press, New York, 1976.

\bibitem{isaacssolvablebook} I.M. Isaacs, \emph{Characters of solvable groups}.  American Mathematical Society, Providence, RI, 2018.

\bibitem{IMNT} I. M. Isaacs, A. Moret\'o, G. Navarro, P. H. Tiep, Groups with just one character degree divisible by a given prime, {\em Trans. Amer. Math. Soc.} {\bf 361} (2009), 6521--6547.

\bibitem{group} I. M. Isaacs, \emph{Finite Group Theory}.  AMS Chelsea Publishing, Providence, RI, 2008.

\bibitem{James} G.~D. James, {\it The representation theory of the symmetric groups}, Lecture Notes in Mathematics, 682, Springer, Berlin, 1978.
 
\bibitem{JK} G.D. James, A. Kerber, {\it The representation theory of the symmetric group}, Encyclopedia of Mathematics and its Applications, 16, Addison-Wesley Publishing Co., Reading, MA, 1981.

\bibitem{KL} P. B. Kleidman, M. W. Liebeck, {\it The subgroup structure of the finite classical groups}, London Mathematical Society Lecture Note Series, 129, Cambridge Univ. Press, Cambridge, 1990.

\bibitem{Luebeck}
{ F. L{\"u}beck},
\newblock Data for finite groups of Lie type and related algebraic groups.
\newblock
\url{http://www.math.rwth-aachen.de/~Frank.Luebeck/chev/index.html?LANG=de.}

\bibitem{Ma07}	
{ G.~Malle}, Height 0 characters of finite groups of Lie type. {\em Represent. Theory \bf11} (2007), 192--220.

\bibitem{Malle} G. Malle, Extensions of unipotent characters and the inductive McKay condition, {\em J. Algebra {\bf 320}} (2008), no.~7, 2963--2980.

\bibitem{MMN} G. Malle, A. Moret\'{o}\ and\ G. Navarro, Element orders and Sylow structure of finite groups, {\em Math. Z. {\bf 252}} (2006), no.~1, 223--230.

\bibitem{MMR}
G. Malle, A. Moret\'o{} and N. Rizo, Minimal heights and defect groups with two character degrees, {\it Adv. Math.} {\bf 441} (2024), Paper No. 109555, 22 pp.

\bibitem{MNST} G. Malle, G. Navarro, A. A. Schaeffer Fry, P. H. Tiep, Brauer's height zero conjecture, {\em Ann. of Math. (2) {\bf 200}} (2024), no.~2, 557--608.

\bibitem{MSF} 
G. Malle, A. A. Schaeffer Fry, On minimal positive heights for blocks of almost quasi-simple groups. {\em Israel J. Math.} To appear. arXiv:2410.22745.

\bibitem{Mann} A. Mann, The derived length of $p$-groups, {\em J. Algebra {\bf 224}} (2000), no.~2, 263--267. 

\bibitem{Maroti} A. Mar\'oti, On elementary lower bounds for the partition function, {\em Integers {\bf 3}} (2003), A10, 9 pp.

\bibitem{MST} A. Moret\'{o}, J. Sangroniz~G\'{o}mez\ and\ A. Turull, Sylow subgroups and the number of conjugacy classes of $p$-elements, {\em J. Algebra {\bf 275}} (2004), no.~2, 668--674.

\bibitem{Navarro_b1} G. Navarro, {\it Characters and blocks of finite groups}, London Mathematical Society Lecture Note Series, 250, Cambridge Univ. Press, Cambridge, 1998. 

\bibitem{Navarro_b2} G. Navarro, {\it Character theory and the McKay conjecture}, Cambridge Studies in Advanced Mathematics, 175, Cambridge Univ. Press, Cambridge, 2018.

\bibitem{Navarro_a1} G. Navarro, Variations on the It\^{o}-Michler theorem on character degrees, {\em Rocky Mountain J. Math. {\bf 46}} (2016), no.~4, 1363--1377.

\bibitem{Navarro_2} G. Navarro, The Eaton-Moret\'o conjecture and $p$-solvable groups, {\em J. Algebra {\bf 665}} (2025), 1-6.

\bibitem{NT} G. Navarro and P. H. Tiep, Characters of relative $p'$-degree over normal subgroups, { \em Ann. of Math.} {\bf 178} (2013), 1135--1171.

\bibitem{NTT} 
G. Navarro, P.~H. Tiep and A. Turull, Brauer characters with cyclotomic field of values, {\em J. Pure Appl. Algebra {\bf 212}} (2008), no.~3, 628--635.

\bibitem{Olssonb} J.~B. Olsson, {\it Combinatorics and representations of finite groups}, Vorlesungen aus dem Fachbereich Mathematik der Universit\"{a}t GH Essen, 20, Universit\"{a}t Essen, Fachbereich Mathematik, Essen, 1993. 


\bibitem{srinivasan}
B. Srinivasan, The characters of the finite symplectic group ${\rm Sp}(4,\,q)$, {\em Trans. Amer. Math. Soc. {\bf 131}} (1968), 488--525.

\bibitem{Suzuki} M. Suzuki, On a class of doubly transitive groups, {\em Ann. of Math. (2) {\bf 75}} (1962), 105--145.


\bibitem{Ward} H.~N. Ward, On Ree's series of simple groups, {\em Trans. Amer. Math. Soc. {\bf 121}} (1966), 62--89.

\bibitem{weir} A.~J. Weir, Sylow $p$-subgroups of the classical groups over finite fields with characteristic prime to $p$, {\em Proc. Amer. Math. Soc. {\bf 6}} (1955), 529--533. 



\end{thebibliography}
\end{document}